\documentclass[ejs,noshowframe]{imsart}

\RequirePackage{amsthm,amsmath,amsfonts,amssymb}
\RequirePackage[authoryear]{natbib}
\usepackage{xcolor}
\definecolor{darkblue}{rgb}{0.0,0.0,0.8}
\RequirePackage[colorlinks,citecolor=darkblue,urlcolor=darkblue,linkcolor=darkblue]{hyperref}
\RequirePackage{graphicx}
\usepackage{bm}
\usepackage{subfigure}
\arxiv{2312.07502}
\startlocaldefs
\theoremstyle{plain}

\newtheorem{theorem}{Theorem}[section]
\newtheorem{lemma}[theorem]{Lemma}
\theoremstyle{definition}

\theoremstyle{remark}

\endlocaldefs

\allowdisplaybreaks

\begin{document}
\begin{frontmatter}
\title{Posterior Concentration for Gaussian Process Priors under Rescaled and Hierarchical Mat\'ern and Confluent Hypergeometric Covariance
Functions}
\runtitle{Posterior concentration for rescaled and hierarchical Mat\'ern and CH}

\begin{aug}
\author{\fnms{Xiao}~\snm{Fang}\ead[label=e1]{fang353@purdue.edu}}
\and
\author{\fnms{Anindya}~\snm{Bhadra}\ead[label=e2]{bhadra@purdue.edu}}
\address{Department of Statistics,
Purdue University, West Lafayette, IN 47907-2066, USA\printead[presep={.\ }]{e1,e2}}
\runauthor{Fang and Bhadra}
\end{aug}

\begin{abstract}
In nonparametric Bayesian approaches, Gaussian stochastic processes can serve as priors on real-valued function spaces. Existing literature on the posterior convergence rates under Gaussian process priors shows that it is possible to achieve optimal or near-optimal posterior contraction rates if the smoothness of the Gaussian process matches that of the target function. Among those priors, Gaussian process with a parametric Mat\'ern  covariance function is particularly notable in that its degree of  smoothness can be determined by a dedicated smoothness parameter.  \citet{ma2022beyond} recently introduced  a new family of covariance functions called the Confluent Hypergeometric (CH) class that  simultaneously possess two parameters: one controls the tail index of the  polynomially  decaying covariance function, and the other parameter controls the degree of  mean-squared smoothness analogous to the Mat\'ern class. In this paper, we show that with proper choice of rescaling parameters in the Mat\'ern and  CH covariance functions, it is possible to obtain the  minimax optimal posterior contraction rate for $\eta$-regular functions for  nonparametric regression model with fixed design. Unlike the previous results for unrescaled cases, the smoothness parameter of the covariance function need not equal $\eta$ for achieving the optimal minimax rate, for either rescaled Mat\'ern or rescaled CH covariances, illustrating a key benefit of rescaling.  We also consider a fully Bayesian treatment of the rescaling parameters and show the resulting posterior distributions  still contract
at the minimax-optimal rate. 
The resultant 
hierarchical Bayesian procedure is  fully adaptive to the unknown true smoothness.   The  theoretical properties of the rescaled and  hierarchical Mat\'ern and CH classes are further verified via extensive simulations and an illustration on a geospatial data set is presented. \\
\end{abstract}

\tableofcontents
\end{frontmatter}

\section{Introduction}
In nonparametric Bayesian estimation approaches, Gaussian processes (GPs) can be adopted as  priors on functional parameters of interest. For
instance, the sample path of a GP can be used to model a real-valued regression function   \citep{kimeldorf1970correspondence,williams2006gaussian}. Moreover, after a monotonic transformation to the unit interval, it can also be used for classification \citep{williams2006gaussian,ghosal2006posterior}. Proceeding further along the same lines, after exponentiation and re-normalization, a GP provides a suitable nonparametric model for density estimation \citep{leonard1978density,tokdar2007posterior}. In all these problems, the study of posterior concentration properties under a Gaussian process prior is of fundamental interest.

To formalize the notation, denote a Gaussian process as: $W=(W_t: t\in T)$ with mean function $\mu(t) = E(W_t)$ and covariance function $K(s,t) = \mathrm{cov}(W_s,W_t )$, $s,t \in T$, where $T$ is an arbitrary index set; such that every finite-dimensional realization of the process admits a multivariate Gaussian distribution with a mean vector and covariance matrix determined by $\mu(\cdot)$ and $K(\cdot, \cdot)$.  Throughout this paper, we consider a zero mean GP, whose properties are completely determined by its covariance function $K(\cdot,\cdot)$. A GP is called (second order) stationary if  the covariance function $K(s,s+h)=C(h)$ is  a function that  only depends on $h$. Further, $C(\cdot)$ is called isotropic if it is a function of $|h|$, where $|\cdot|$ denotes the Euclidean norm.

Among the parametric family of covariance functions, the isotropic Mat\'ern model is popular and is a good default choice \citep{stein1999interpolation,porcu2023mat}. A key reason for the popularity of Mat\'ern is that there is a dedicated parameter controlling the degree of mean-squared smoothness of the associated random process. However, the Mat\'ern class possesses an exponentially decaying tail, which is unsuitable if distant observations are highly correlated; a situation that is better captured by polynomially decaying covariances.  \citet{ma2022beyond} recently introduced a new family of covariance functions called the Confluent Hypergeometric (CH) class by using a scale mixture representation of the Mat\'ern class. The main motivation behind the CH covariance function is that it possesses polynomial decaying tails, unlike the exponential tails of the Mat\'ern class. Moreover, a key benefit of the CH class, unlike other polynomial covariances such as the generalized Cauchy but like Mat\'ern, is that it possesses a dedicated parameter controlling the degree of mean-squared differentiability of the associated Gaussian process \citep{stein1999interpolation}. In this sense, the CH class combines the best properties of Mat\'ern and polynomial covariances.  Throughout, we use \emph{Mat\'ern process} as a shorthand for a GP with a Mat\'ern covariance function, and similarly for other covariance models.

Given a specification of prior and likelihood, an application of Bayes' rule  yields a posterior distribution. It is of fundamental interest to study the contraction rates of such Bayesian posteriors, i.e., the rate at which the posterior distribution contracts around the true unknown functional parameter of interest.
There exists a substantial literature on the posterior contraction rates of Gaussian processes in the Bayesian framework; see for example \citet{van2007bayesian,van2008rates,van2011information,castillo2008lower,castillo2014bayesian,giordano2020consistency,nickl2017nonparametric, nickl2023bayesian,pati2015optimal,van2016gaussian} and references therein, with a textbook level detailed exposition available in \citet{ghosal2017fundamentals}. These works  reveal that priors based on Gaussian processes lead to optimal or near-optimal posterior contraction rates, provided the smoothness of  the Gaussian process matches that of the target function. Both oversmoothing and undersmoothing lead to suboptimal contraction rates. For example, for  $\eta$-regular  target functions (see Section~\ref{section2.1} for a formal definition), the smooth  squared exponential process, i.e. the centered Gaussian process $W$ with covariance function $C(h)=a \exp(-b |h|^2)$ for some $a,b>0$, yields a very slow posterior contraction rate $(1/\log(n))^{\theta}$ for some positive constant $\theta$, and the Mat\'ern process attains the optimal minimax rate only when its smoothness parameter equals the function regularity $\eta$  \citep{van2011information}. A key reason for this is that squared exponential processes lead to realizations that are infinitely differentiable in the mean squared sense, i.e., very smooth. Hence, a squared exponential process  is not appropriate for modeling a functional parameter with some finite smoothness level (e.g., belonging to a Sobolev space), and yields  very slow posterior contraction. Similarly, the Mat\'ern class also leads to suboptimal rates if the roughness of the true function does not match the degree of mean-squared differentiability of the covariance function.

\cite{van2007bayesian} remedy this  problem by suitably rescaling the smooth process under a squared exponential covariance, with rescaling constants depending on the sample size, in the following sense.  Consider a prior
process $t \to W_t^c:= W_{t/c}$  for some $c>0$, where the parameter $c$ can be thought of changing the lengthscale of the process. If the scale parameter $c$ is limited to a compact subset of $(0,\infty)$, then
the contraction rate does not change \citep{van2008rates}. However, while the  smoothness of the sample path does not change for any fixed $c$, a dramatic impact can be observed on the posterior contraction rate  when $c = c_n$ decreases to $0$ or increases to infinity as the sample size $n$ goes to infinity.  Shrinking with $c$ (i.e., the $c<1$ case) can make a given process arbitrarily rough. By this technique, \cite{van2007bayesian} successfully improve the posterior contraction rate for the squared exponential process to the optimal minimax rate (up to a logarithmic factor) for $\eta$-regular functions. Similar ideas for rescaling have appeared in other works related to Gaussian processes \citep{pati2015optimal,10.1214/20-AOS2043}. However, these works deal with  Gaussian processes with a squared exponential covariance.  In this paper, we address the issue of posterior concentration under  the CH process prior, as well as the Mat\'ern process prior with suitable \emph{rescaling}, which has remained unaddressed.   For the isotropic Mat\'ern class, the covariance
function has the form \citep{williams2006gaussian}:
\begin{eqnarray}
M(h;v,\phi,\sigma^2)=\sigma^2 \frac{2^{1-v}}{\Gamma(v)}\left(\frac{\sqrt{2v}}{\phi} h \right)^v {K}_v\left( \frac{\sqrt{2v}}{\phi} h \right);\; v>0,\, \phi>0,\, \sigma^2>0,\label{eq:Matern}
\end{eqnarray}  
where $K_v(\cdot)$ is the modified Bessel function of the second kind \citep[][Section 9.6]{abramowitz1988handbook}. We observe that the parameter $\phi$ is the lengthscale parameter, and is a natural candidate for rescaling. For the isotropic CH class of \cite{ma2022beyond}, the covariance function is:
\begin{eqnarray}
    C(h;v,\alpha,\beta,\sigma^2)=\frac{\sigma^2 \Gamma(v+\alpha)}{\Gamma(v)}U\left(\alpha,1-v,v\left(\frac{h}{\beta}\right)^2\right),\label{eq:CH}
\end{eqnarray}    
where $U(a,b,c)$ is the confluent hypergeometric function of the second kind, defined as in \citet[][Section 13.2]{abramowitz1988handbook}: 
\begin{displaymath}
   U(a,b,c) :=\frac{1}{\Gamma(a)} \int_{0}^{\infty} e^{-ct}t^{a-1}(1+t)^{b-a-1}dt;\; a>0,\, b\in\mathbb{R},\, c>0.
\end{displaymath}
If $\alpha$ is fixed, then the parameter $\beta$ is the lengthscale parameter and is a natural candidate for rescaling.  We control the smoothness of the Gaussian process by changing $\phi$  for the Mat\'ern class  and $\beta$ for the CH class. The key to achieving  the  optimal posterior contraction rates for Mat\'ern and CH  classes lies in  appropriately choosing the rescaling parameters when the true unknown functional parameter of interest is rougher than the mean-squared differentiability of a given covariance function. Indeed, by rescaling $\phi$ in the Mat\'ern class and by rescaling the parameters $\beta$ in the CH class, we obtain optimal minimax  posterior contraction  rate under both priors  for $\eta$-regular true functions, and our posterior contraction rates do not include the logarithmic factor as in \cite{van2007bayesian}.  We note here  \cite{giordano2020consistency} and \cite{nickl2023bayesian} also consider rescaled and undersmoothed $\alpha$-regular processes, which include Mat\'ern processes,  in the context of Bayesian inverse problems. However,  their settings are different from ours, in that they focus on posterior contraction performance  under  their forward map.  

The rescaling approach developed above depends explicitly on the
regularity of the true function $\eta$, which is  typically unknown in practice. To fully address this limitation, we assign  priors on the rescaling parameter  as in \cite{van2009adaptive}, to develop a fully Bayesian alternative,  and show that under this procedure the optimal minimax rate can be achieved simultaneously over a range of values for the true regularity.  Estimators that are rate optimal for a range of regularity levels have been called  \textit{adaptive} \citep{efroimovich1984learning,lepskii1991problem,lepskii1992asymptotically}.  Consequently, our contributions also lie in designing adaptive posterior concentration results for Mat\'ern and CH processes, resulting in a practically useful procedure.

The  remainder of the paper is organized as follows. In Section \ref{section2}, we provide some  relevant background on posterior contraction rates for Gaussian process priors.  Section \ref{section3} presents our main theorems on  posterior contraction rates for  rescaled Mat\'ern  and CH process priors, and the fully Bayesian adaptive versions over a range of regularity values. An extension to the anisotropic  case is discussed in Section~\ref{section4}. In Section \ref{section5}, we  compare the rescaled and hierarchical CH, Mat\'ern, and squared exponential process priors via simulations.  Analysis of a spatial data set is presented in Section~\ref{section6}.  Section \ref{section7} concludes with some discussions for future investigations. Mathematical proofs of all results and further technical details can be found in the Appendix.

\section{Preliminaries on Posterior Contraction under Gaussian Process Priors} \label{section2}
\subsection{Notation and the   Space of $\eta$-regular Functions} \label{section2.1}
For two positive sequences $\{a_n\}, \{b_n\}$, we denote by $a_n \lesssim b_n$ that $a_n=O(b_n)$, and by $a_n \gtrsim b_n$ that $b_n=O(a_n)$, with $a_n \asymp b_n$ denoting $a_n \lesssim b_n$ and $a_n \gtrsim b_n$ simultaneously. We use $m_M^\phi$ and $m_{CH}^{\alpha,\beta}$ to denote  the spectral density of Mat\'ern  and CH process, and their exact expressions are presented in Appendix \ref{Ancillary}.

The following notations are similar as in \citet{van2011information}, but we summarize them here for the ease of  reference. For $\eta>0$, let $\eta=m+\xi$, for $\xi \in(0,1]$ and $m$ a nonnegative integer. For  $T \subset \mathbb{R}^d$, the H\"older space $C^\eta(T)$ is the space of functions whose partial derivatives of orders $\left(k_1, \ldots, k_d\right)$ exist for nonnegative integers $k_1, \ldots, k_d$ with $k_1+\ldots+k_d \leq m$ and the highest order partial derivatives which  are Lipschitz  are of order $\xi$. A function $f$ is said to be  Lipschitz of order $\xi$ if $|f(x)-f(y)| \leq$ $L\|x-y\|^\xi$, for every $x, y \in T$ and $ L>0$. We denote by $C(T)$  the space of all continuous functions on $T$. 

Let $L_2(\mu)$ denote the set of all functions which are square integrable with respect to measure $\mu$.

The Sobolev space $H^\eta(\mathbb{R}^d)$ is the set of functions 
$f_0:\mathbb{R}^d \rightarrow \mathbb{R}$ such that: 
$$
\left\|f_0\right\|_{2,2,\eta}^2:=\int\left(1+\|\lambda\|^2\right)^\eta\left|\hat{f}_0(\lambda)\right|^2 d \lambda<\infty,
$$
where $\hat{f}_0(\lambda)=(2 \pi)^{-d} \int e^{-i <\lambda,t>} f_0(t) dt$ 
is the Fourier transform of $f_0$. For $T \subset \mathbb{R}^d$, the Sobolev space $H^\eta(T)$  is the set of functions  $w_0:T \rightarrow \mathbb{R}$ that are restrictions of a function $f_0: \mathbb{R}^d \rightarrow \mathbb{R}$ in $H^\eta(\mathbb{R}^d)$. A function $f: T \rightarrow \mathbb{R}$ is called $\eta$-regular on $T$ if $f \in C^\eta(T) {\cap} H^\eta(T)$.

For $x_1, \cdots, x_n\in T$ and a function $w:T \to \mathbb{R}$, we define the empirical norm   $\|w\|_n$ by: 
$$
\|w\|_n=\left( \frac{1}{n} \sum_{i=1}^n w^2(x_i)\right)^{1/2}.
$$

A bounded domain $\mathcal{X} \subset \mathbb{R}^d$ is said to be Lipschitz if at each point of its boundary, it is locally the set of points located above the graph (i.e., an epigraph) of some Lipschitz function; for a more formal definition, see \citet[][p.~227]{wellner2013weak}.  In this section, and throughout the remainder of the article, $\mathcal{T}$ denotes a convex bounded Lipschitz domain in $\mathbb{R}^d$.

\subsection{Posterior Contraction Rates for Gaussian Process Priors} \label{section2.3}
In this section, we state the necessary background on posterior contraction rates for Gaussian process priors developed by \cite{van2008rates}, who show that  for a mean zero Gaussian process prior $W$, if a functional parameter of interest $w_0$ is in the closure of the reproducing kernel Hilbert space (RKHS) of
this process,  the rate of convergence  at $w_0$ is determined by  its concentration function, defined as:
\begin{equation}
\varphi_{w_0}(\varepsilon_n)= \inf_{h \in {\mathbb{H}}:\|h-w_0\|\le\varepsilon_n}\|h\|_{\mathbb{H}}^2 -\log P(\| W\| \le \varepsilon_n), 
\end{equation}
where ${\mathbb{H}}$ is the RKHS of the process $W$, $\|.\|_{\mathbb{H}}$ is the RKHS-norm  and $\|\cdot\|$ is the norm of the Banach space in which
 $W$ takes its values. By Theorem 2.1 of \cite{van2008rates}, we get the conditions needed to apply the general results on  posterior contraction rates  as stated in Theorem 2.1 of \cite{ghosal2000convergence} by solving:
\begin{equation}
    \varphi_{w_0}(\varepsilon_n) \le n \varepsilon_n^2.\label{concentration function}
\end{equation}
 One may note that Theorem 2.1 of  \cite{van2008rates} uses the Banach space norm, whereas the general conditions for posterior contractions of \citet{ghosal2000convergence} may use other appropriate statistical distances. Nevertheless, the rate of contraction $\epsilon_n$
is obtained when these metrics  are comparable to the  Banach norm \citep[p.~1439,][]{van2008rates}.

In the current paper, we consider the nonparametric  regression  model with fixed design, taking values in $C(T)$, $T\subset \mathbb{R}^d$ and $C(T)$ is a Banach space equipped with  the supremum norm $\|\cdot\|_{\infty}$. In Section \ref{section2.4} we demonstrate how the concentration function determines the posterior contraction rate in this model. A Mat\'ern process takes its values in $C^{v\prime}(T)$ for any $v^\prime < v$ \citep[p.~2104,][]{van2011information}. Hence it  also takes value in $C\left(T\right)$. Following a very similar argument, the sample paths of the CH process $W_{CH}$ have the same smoothness in $L_2$ as the functions $e_t(\lambda)=e^{i \lambda^T t}$ in $L_2(m^{\alpha,\beta}_{CH})$. The sample paths are $k$ times differentiable in $L_2$ \citep{ma2022beyond}, where $k$ is the integer part of the smoothness parameter $v$ for the CH process, with the $k$-th derivative $W^{(k)}_{CH}$ satisfying for $s,t\in T$:
$$
\mathrm{E}\left(W_{CH,s}^{(k)}-W_{CH,t}^{(k)}\right)^2 \lesssim\|s-t\|^{2(v-k)}.
$$
Hence, by an argument analogous to \citet{van2011information}, the CH process takes its values in $C^{{v}'}(T)$ for any ${v'}<v$. Hence, it  also takes value in $C\left(T\right)$.
 
The required conditions for posterior contraction can be further decomposed into  the following  pair of inequalities:
\begin{equation}
\varphi_0(\varepsilon_n)=-\log P(\| W\| \le \varepsilon_n) \le \frac{1}{2}n \varepsilon_n^2, \qquad
\inf_{h \in {\mathbb{H}}:\|h-w_0\|\le\varepsilon_n}\|h\|_{\mathbb{H}}^2 \le \frac{1}{2}n \varepsilon_n^2. \label{two parts}
\end{equation}
The final rate $\varepsilon_n$  can be obtained by  solving the two inequalities in (\ref{two parts}) simultaneously and taking the maximum of the two  solutions. 

Some further insight into these inequalities can be obtained as follows. The first inequality in (\ref{two parts}) deals with  the small ball probability at 0, i.e., the prior mass around zero. It depends only on the prior,
but not on the true parameter $w_0$.  Priors that put more mass near 0 tend to give quick rates $\varepsilon_n$, yielding a strong shrinkage effect towards zero for all functions. The second inequality measures how well $w_0$ can be approximated by elements in the RKHS of the prior, the ideal case being that $w_0$ is contained in the RKHS.  If we take $h=w_0$, then the infimum is bounded by $\left\|w_0\right\|_{\mathbb{H}}^2$, showing that $\varepsilon_n$ must not be smaller than  the \emph{parametric rate} $n^{-1/2}$. In sum, to obtain quick rate $\varepsilon_n$,  the prior should put sufficient mass around 0, and  the true parameter $w_0$ should be in the RKHS, or needs to be well approximated by elements in the RKHS (since  the RKHS can be a very small space, assuming $w_0$ belongs to it may be too strong an assumption). Whether a balance could be struck between these two disparate goals in \eqref{two parts} determines the posterior concentration properties. Moreover, it can also be shown \citep{van2008reproducing} that up to constants, $\varphi_{w_0}(\varepsilon)$ equals $-\log P(\|W-w_0\|<\varepsilon)$, so the rate of contraction of the true function is completely determined by the prior mass around the truth.

\subsection{Nonparametric Regression with Fixed Design and Additive Gaussian Errors} \label{section2.4}
In the current  work we assume that given a  deterministic function $w:T \to \mathbb{R}$, the data $Y_1, \ldots ,Y_n$ are independently generated by $Y_j = w(x_j)+\varepsilon_j$, for fixed, known $x_j \in T$ and independent $\varepsilon_j \sim N(0,\sigma_0^2)$, with $\sigma_0$ known and fixed. A prior on $w$ is induced by setting $w(x) = W_x$, for a Gaussian process $(W_x : x\in T)$. Then $w$ can be treated as the sample function of the Gaussian process.

By Theorem 1 in \cite{van2011information}, for $w_0 \in C_b(T)$, where $C_b(T)$ is the set of bounded, continuous functions on the compact metric space $T$, one has:
\begin{equation}
E_{w_0} \int \|w-w_0\|_n^2 d \Pi_n(w\mid Y_1,\ldots,Y_n) \lesssim \Psi^{-1}_{w_0}(n)^2,\label{eq:van}
\end{equation}
where $\Psi_{w_0}(\varepsilon)=\frac{\varphi_{w_0}(\varepsilon)}{\varepsilon^2}$, the Banach norm in the concentration function is the supremum norm $\|\cdot\|_{\infty}$ and $\Psi_{w_0}^{-1}({l})=\sup\{\varepsilon>0:\Psi_{w_0}(\varepsilon)\ge l\}$, which shows that the posterior distribution contracts at the rate $\Psi^{-1}_{w_0}(n)$ around the true response function $w_0$. 

\section{Posterior Contraction Rates for Isotropic Cases}\label{section3}
In this section we study Gaussian process priors with rescaled isotropic Mat\'ern  and CH covariance functions. Section \ref{rescaled measure section} introduces results describing their RKHSs. In  Sections \ref{section3.2} and \ref{section3.3}, we obtain results
illustrating their small deviation behavior and the approximation properties of their
RKHSs. Minimax optimal rates of convergence for the respective posteriors are obtained by  applying the general theory of Section \ref{section2.3} to the nonparametric regression with fixed design described in Section \ref{section2.4}.
In Section \ref{section3.4}, we discuss the hierarchical  Mat\'ern  and CH process priors and show these hierarchical Bayesian procedures also yield minimax optimal rates of convergence, over a range of regularity values for the true function.


\subsection{RKHSs of Rescaled Stationary Gaussian Processes}\label{rescaled measure section}
We consider  a mean zero stationary Gaussian process $W=(W_t: t \in T)$ with covariance function $K(s,s+h)=C(h)$, where $T \subset \mathbb{R}^d$.
By Bochner's theorem, the function $C(\cdot)$ is representable as the characteristic function
$
C(t)=\int e^{-i <\lambda,t>} d \mu(\lambda),
$
of a symmetric, finite measure $\mu$ on $\mathbb{R}^d$, termed the spectral measure of the process $W$.   By Lemma 4.1 of \cite{van2009adaptive}, the RKHS of a stationary Gaussian  process $W$ is  the space of all (real parts of) functions of the form:
\begin{equation}
(\mathcal{F}\psi)(t)= \int e^{{i} <\lambda, t>}\psi(\lambda)  d\mu(\lambda), \label{fourier trans}
\end{equation}
where $\psi$ ranges over $L_2(\mu)$, and  the squared RKHS-norm is given by:
\begin{equation}
    \|\mathcal{F} \psi\|_{\mathbb{H}}^2  =\inf _{g: \mathcal{F}g=\mathcal{F}\psi} \int|g|^2(\lambda)  d\mu (\lambda). \label{rkhs norm}
\end{equation}
The infimum is unnecessary if the spectral density has exponential or lighter tails, but is necessary in our case.

Now we define the rescaled version $W^c$ of the process $W$ by setting $W_t^c=W_{t / c},\; c>0$, with $W$ denoting the process with $c=1$.

 Following  \cite{van2007bayesian}, the spectral measure $\mu_c$ of the rescaled process $W^c$ is obtained by rescaling the spectral measure $\mu$ of $W$ as:
$$
\mu_c(B)=\mu(c B),
$$
where $B$ is any Borel set with respect to $\mu$.
Denote by $\mathcal{F}_c h$ the transform $\mathcal{F}_c h: \mathbb{R}^d \rightarrow \mathbb{C}$ of the function $h \in L_2\left(\mu_c\right)$:
\begin{equation}\label{rescaled matern map}
\left(\mathcal{F}_c h\right)(t)=\int e^{i <\lambda,t>} h(\lambda) d \mu_c(\lambda).
\end{equation}
Then $\mathcal{F}_c$ maps $L_2\left(\mu_c\right)$ into the space $C(\mathbb{R}^d)$ \citep{van2007bayesian}. 

For the Mat\'ern  class, let $W^{\phi}$ be the process with Mat\'ern covariance function having parameter $\phi$ as in~\eqref{eq:Matern}, and $W_t^{\phi}=W_{t/\phi}$, this  means that the Mat\'ern-$\phi$ process has the interpretation of a Mat\'ern-1 process whose sample paths are rescaled by $\phi$. Then $\phi$ is the scale parameter, and we can define the rescaled spectral measure $\mu_\phi$ and transform $\mathcal{F}_\phi$ as before. For the CH class, let $W^{\alpha, \beta}$ be the process with the CH covariance function  (\ref{eq:CH}) having parameters $\alpha$ and $\beta$. If both $\alpha$ and $\beta$ are free to vary (with sample size $n$), we can not find process $\tilde{W}$ and $c$, such that  $W^{\alpha, \beta}_t=\tilde W_{t/c}$, so we can not define the rescaled spectral measure as in the Mat\'ern case. Similar to setting $d \mu_{\phi}(\lambda) = m^{\phi}_M d \lambda $  in the Mat\'ern  case, for CH class, we  set $d \mu_{\alpha,\beta}(\lambda)=m^{\alpha,\beta}_{CH} d\lambda$ and denote by $\mathcal{F}_{(\alpha,\beta)} h$ the transform $\mathcal{F}_{(\alpha,\beta)}  h: \mathbb{R}^d \rightarrow \mathbb{C}$ of the function $h$:
\begin{equation}\label{rescaled CH map}
\left(\mathcal{F}_{(\alpha,\beta)} h\right)(t)=\int e^{i <\lambda,t>} h(\lambda)  d \mu_{\alpha,\beta}(\lambda).
\end{equation}

The following lemma describes the RKHS $\mathbb{H}^{\phi}$ of the process $\left(W_t^{\phi}: t \in \mathcal{T}\right)$ and RKHS $\mathbb{H}^{\alpha,\beta}$ of the process $\left(W_t^{\alpha,\beta}: t \in \mathcal{T}\right)$. We also denote the unit ball in $\mathbb{H}^{\phi}$ by $\mathbb{H}_1^{\phi}$ and the unit ball in $\mathbb{H}^{\alpha,\beta}$ by $\mathbb{H}_1^{\alpha,\beta}$.
\begin{lemma} \label{rescaled measure}
   If  $W$ is a centered stationary Gaussian process with Mat\'ern covariance function~\eqref{eq:Matern}, the RKHS of the  process $\left(W_t^{\phi}:  t \in  \mathcal{T}\right)$ is the set of real parts of all transforms $\mathcal{F}_{\phi} h$ (restricted to  $ \mathcal{T} \subset \mathbb{R}^d$) of functions $h \in L_2\left(\mu_{\phi}\right)$, equipped with the square norm:
\begin{equation}\label{RKHS of rescaled matern}
\left\|\mathcal{F}_{\phi} h\right\|_{\mathbb{H}^{\phi}}^2 =\inf _{g: \mathcal{F}_{\phi} g=\mathcal{F}_{\phi} h} \|g\|_{L_2\left(\mu_{\phi}\right)}^2=\inf _{g: \mathcal{F}_{\phi} g=\mathcal{F}_{\phi} h}\int|g|^2(\lambda) d \mu_{\phi}(\lambda) .
\end{equation}
For centered Gaussian process with CH covariance function (\ref{eq:CH}),  the RKHS of process $\left(W_t^{\alpha,\beta}: t \in  \mathcal{T}\right)$  is the set of real parts of all transforms $\mathcal{F}_{(\alpha,\beta)} h$ (restricted to  $ \mathcal{T} \subset \mathbb{R}^d$) of functions $h \in L_2\left(\mu_{\alpha,\beta}\right)$, equipped with the square norm:
\begin{equation}\label{RKHS of rescaled CH}
\left\|\mathcal{F}_{\alpha,\beta} h\right\|_{\mathbb{H}^{\alpha,\beta}}^2 =\inf _{g: \mathcal{F}_{(\alpha,\beta)}  g=\mathcal{F}_{(\alpha,\beta)}  h} \|g\|_{L_2\left(  \mu_{\alpha,\beta}\right)}^2=\inf _{g: \mathcal{F}_{(\alpha,\beta)} g=\mathcal{F}_{(\alpha,\beta)}  h}\int|g|^2(\lambda) d \mu_{\alpha,\beta}(\lambda) .
\end{equation}
\end{lemma}
The proof is a direct consequence of  Lemma 4.1 of \cite{van2009adaptive} and is therefore omitted. 
\subsection{Posterior Contraction Rates for the Rescaled Mat\'ern Class}\label{section3.2}
The following lemma studies the small ball probability of the rescaled Mat\'ern class. We establish this  lemma by the fact that the small ball exponent can be obtained from the metric entropy of  unit ball $\mathbb{H}_1$ of the RKHS for the Gaussian process $W$ \citep{li1999approximation}. In our proof, we also show that the RKHS  of the rescaled Mat\'ern class is approximately a Sobolev space $H^{v+d/2}(\mathcal{T})$, with a rescaling factor. 
\begin{lemma}
Suppose $\phi < 1$. There exists an $\varepsilon_0>0$, independent of $\phi$, such that 
the small ball exponent of the rescaled  centered Mat\'ern 
 process $W^\phi$ with  covariance function  (\ref{eq:Matern}) satisfies,
$$
\varphi_0(\varepsilon) = -\log P(\| W^\phi\|_{\infty} \le \varepsilon)= - \log P(\sup_{t \in \mathcal{T}} |W_t^\phi| \le \varepsilon) 
 \lesssim \varepsilon^{-d/v} \phi^{-d},
 $$
for $\varepsilon \in (0,\varepsilon_0)$.
 \label{matern small ball}
\end{lemma}

The following lemma quantifies how well $\eta$-regular functions can be approximated by elements in the RKHS of the rescaled Mat\'ern process. Appealing to \cite{van2007bayesian}, we introduce parameter $\theta>1$ to be determined. This parameter is crucial in our proof, and with larger $\theta$  we have better RKHS approximation performance, while with smaller $\theta$ we have smaller small ball exponent. By tuning $\theta$,  we balance small ball exponent  and decentering parts, and obtain the minimax optimal posterior contraction rate.

\begin{lemma}
Suppose $w_0 \in C^{\eta}(\mathcal{T}) \cap H^{\eta}(\mathcal{T})$. Suppose the smoothness parameter $v$ of   rescaled centered Mat\'ern process $W^{\phi}$  with  covariance function  (\ref{eq:Matern})  satisfies  $ v \ge \eta >0$.  Then for $\theta > \frac{2v+d}{2v+d-2 \eta}$, we have:
$$
\inf_{h \in \mathbb{H}^{\phi}:\|h-w_0\|_{\infty}\le C_{w_0}\phi^{\theta \eta}}\|h\|_{\mathbb{H}^{\phi}}^2 \le D_{w_0}\phi^{2v-2\theta(v+d/2-\eta)},
$$
as $\phi \downarrow 0$, where $C_{w_0}$,$D_{w_0}$ only depend on $w_0$.
 \label{matern decentering}
\end{lemma}
Now  combining the two preceding lemmas, for  $w_0 \in C^{\eta}(\mathcal{T}) \cap H^{\eta}(\mathcal{T})$ with $ \eta \leq v$, we obtain the following inequalities:
\begin{displaymath}
    \varepsilon_n^{-d/v} \phi^{-d} \lesssim n \varepsilon_n^2 , \quad \phi^{2v-2\theta(v+d/2-\eta)} \lesssim n \varepsilon_n^2, \quad\phi^{\theta \eta} \lesssim  \varepsilon_n,  \quad \theta > \frac{2v+d}{2v+d-2 \eta}.
\end{displaymath}
It suffices to solve:
\begin{displaymath}
   \varepsilon_n \ge \left( \frac{\phi^{2v}}{n}\right)^{\frac{\eta}{2v+d}}, \quad
   \varepsilon_n \ge \left( \frac{\phi^{-d}}{n}\right)^{\frac{v}{2v+d}}, 
\end{displaymath}
which leads to $\varepsilon_n \gtrsim n^{-\frac{\eta}{2\eta+d}}$, with equality attained when $\phi=n^{-\frac{v-\eta}{(2\eta+d)v}}$. Then by    an  application of \eqref{eq:van} in Section \ref{section2.4}, we obtain the following theorem. 

\begin{theorem}
Suppose we use a centered Mat\'ern prior  with  covariance function  (\ref{eq:Matern}),  $0<\phi<1$,  $w_0 \in C^{\eta}(\mathcal{T}) \cap H^{\eta}(\mathcal{T})$ and $v \ge \eta>0$.  If $\phi=n^{-\frac{v-\eta}{(2\eta+d)v}}$,  then  for nonparametric regression with fixed design and additive Gaussian errors, 
$$
E_{w_0} \int \|w-w_0\|_n^2 d \Pi_n(w\mid Y_1,\ldots,Y_n) \lesssim (n^{-\frac{\eta}{2\eta+d}})^2,
$$
i.e. the posterior contracts at the rate $n^{-\frac{\eta}{2\eta+d}}$.  \label{rescale matern theorem}
\end{theorem}   

For $w_0$  defined on a compact subset of $\mathbb{R}^d $ with regularity $\eta> 0$, it is known $\varepsilon_n = n^{-\frac{\eta}{2 \eta+d}}$ is the minimax-optimal rate \citep{tsybakov2009introduction,stone1980optimal}. It follows that this is also the best possible bound for the risk in Section \ref{section2.4} if $w_0$ is a $\eta$-regular function of $d$ variables. Thus, in Theorem \ref{rescale matern theorem}, we have obtained minimax optimal rate.  \cite{van2008rates} show that for GP priors, it is typically true that
this optimal rate can only be attained if the regularity of the GP that is used matches the regularity
of $w_0$. Using a GP prior that is too rough or too smooth
harms the performance of the procedure. Compared to the Mat\'ern process prior with fixed scale parameter, which only obtains minimax optimal rate in the $v=\eta$ case \citep{van2011information}, our theorem extends to the case $v> \eta$. This is because by rescaling the parameter $\phi$, we successfully match the smoothness of  the Mat\'ern process prior to  $w_0$. Compared to the rescaled squared exponential  prior of \cite{van2007bayesian}, our theorem obtains the minimax optimal rate while their rate is   minimax optimal up to a logarithmic factor. A possible explanation is that the squared exponential process is infinitely smooth and Mat\'ern is finitely differentiable, even after rescaling. Thus, a rescaled Mat\'ern  prior can still capture a rough function better.

\cite{castillo2008lower} studies the lower bound of posterior contraction rate, and  finds it is determined by the concentration function $\varphi_{w_0}(\varepsilon_n)$. Larger concentration function implies slower contraction rate. For  $\mathcal{T}=[0,1]$, we  observe  when $\phi$ goes to $0$ very quickly, the sample path of $W_t^\phi$ shrinks into the interval $[0,1]$, and intuitively,  the small ball part of the concentration function
$\varphi_0(\varepsilon) = - \log P(\sup_{t \in \mathcal{T}} |W_t^\phi| \le \varepsilon)$ 
goes to infinity quickly. This slows down the posterior contraction rate and leads to a suboptimal rate. Under suitable conditions, the posterior even  fails to contract around the truth. The following theorem validates this observation for  the case when $\mathcal{T}$ is a convex bounded Lipschitz domain in $\mathbb{R}^d$.
\begin{theorem}
Suppose we use a centered Mat\'ern prior  with  covariance function  (\ref{eq:Matern}).  Then  for nonparametric regression with fixed design and additive Gaussian errors, we have,
$$
  \varphi_0(\varepsilon )\gtrsim \phi^{-d} \varepsilon^{-d/v}. 
$$
When $v > \eta$ and $\phi = o(n^{-\frac{v-\eta}{(2\eta+d)v}})$, the posterior contraction rate is suboptimal. Furthermore, when $\phi^{-d} \gtrsim n$, 
$$
 \Pi_n(w:\|w-w_0\|_n \leq 1 \mid Y_1,\ldots,Y_n)  \to 0,
$$
in probability $P_0^n$, i.e., the posterior does not contract.
\label{lower bound of rescaled matern}
\end{theorem}

\subsection{Posterior Contraction Rates for the Rescaled CH Class}\label{section3.3}
In this subsection we show the rescaled CH and Mat\'ern  classes have similar posterior contraction behavior, which can be expected because the tails of the respective spectral densities only differ by a slowly varying function \citep{ma2022beyond}, and  the regularity of functions $\mathcal{F}_{\phi} \psi$ in RKHS is determined by the tails of the spectral measure \citep[][Chapter 11.4.4]{ghosal2017fundamentals}.

\begin{lemma} \label{CH small ball}
 Suppose $\frac{\Gamma(\alpha+v)}{\Gamma(\alpha) \beta^{2v}} >1 $, $\alpha>d/2+1$. There exists an $\varepsilon_0>0$, independent of $\alpha$ and $\beta$, such that 
the small ball exponent of the rescaled centered CH
 process $W^{\alpha,\beta}$  with  covariance function  (\ref{eq:CH}) satisfies,
$$
\varphi_0(\varepsilon) = -\log P(\| W^{\alpha,\beta}\|_{\infty} \le \varepsilon)= - \log P(\sup_{t \in \mathcal{T}} |W_t^{\alpha,\beta}| \le \varepsilon) 
 \lesssim \varepsilon^{-d/v} \left(\frac{\Gamma(\alpha+v)}{\Gamma(\alpha) \beta^{2v}}\right)^{\frac{d}{2v}},
$$
for $\varepsilon \in (0,\varepsilon_0)$.
\end{lemma}

\begin{lemma}\label{CH decentering}
Suppose $w_0 \in C^{\eta}(\mathcal{T}) \cap H^{\eta}(\mathcal{T})$. If the smoothness parameter $v$ for centered Gaussian process $W^{\alpha,\beta}$  with  covariance function  (\ref{eq:CH}) satisfies  $ v \ge \eta >0$,  then  for  $\alpha>d/2+1$, $\alpha \leq C \sqrt{\ln \ln{n}}$  for sufficient large $n$ and an arbitrarily large multiplicative constant $C$ that  does not depend on $n$, $\beta \lesssim \ln n$ and $\theta > \frac{2v+d}{2v+d-2 \eta},$ we have: 
$$
\inf_{h \in \mathbb{H}^{\alpha,\beta}:\|h-w_0\|_{\infty}\le C_{w_0} \beta^{\theta \eta}}\|h\|_{\mathbb{H}^{\beta}}^2 \le D_{w_0}(\beta^{2\theta})^{-v-d/2+\eta} \frac{\Gamma(\alpha) \beta^{2v}}{\Gamma(\alpha+v)},
$$
as $\beta \downarrow 0$, where $C_{w_0}$,$D_{w_0}$ only depend on $w_0$.
\end{lemma}
Now, combine the two lemmas before and  solve the following inequalities:
\begin{displaymath}
    (\beta^{2\theta})^{-v-d/2+\eta} \frac{\Gamma(\alpha) \beta^{2v}}{\Gamma(\alpha+v)} \lesssim  n \varepsilon_n^2, \quad \beta^{\theta \eta} \lesssim \varepsilon_n,\quad  \varepsilon_n^{-d/v} \left(\frac{\Gamma(\alpha+v)}{\Gamma(\alpha) \beta^{2v}}\right)^{d/(2v)} \lesssim n \varepsilon_n^2.
\end{displaymath}
It suffices to solve:
\begin{displaymath}
   \varepsilon_n \ge \left( \frac{\beta^{2v}}{n}\right)^{\frac{\eta}{2v+d}}, \quad
   \varepsilon_n \ge \left( \frac{\beta^{-d}}{n}\right)^{\frac{v}{2v+d}}, 
\end{displaymath}
leading to  $\varepsilon_n \gtrsim n^{-\frac{\eta}{2\eta+d}}$, with equality attained when $\beta=n^{-\frac{v-\eta}{(2\eta+d)v}}$. Combining this rate and an application of \eqref{eq:van} in Section  \ref{section2.4}, we obtain the following theorem.
\begin{theorem}
Suppose we use a centered CH prior with  covariance function  (\ref{eq:CH}), $\alpha>d/2+1$ and $\alpha \leq C \sqrt{\ln \ln{n}}$  for sufficiently large $n$ and an arbitrarily large multiplicative constant $C$ that  does not depend on $n$, $w_0 \in C^{\eta}(\mathcal{T}) \cap H^{\eta}(\mathcal{T})$ and $v \ge \eta>0$.  If $\beta=n^{-\frac{v-\eta}{(2\eta+d)v}}$,  then  for nonparametric regression with fixed design and additive Gaussian errors,
$$
E_{w_0} \int \|w-w_0\|_n^2 d \Pi_n(w\mid Y_1,\ldots,Y_n) \lesssim (n^{-\frac{\eta}{2\eta+d}})^2,
$$
i.e. the posterior contracts at the (minimax optimal) rate $n^{-\frac{\eta}{2\eta+d}}$.
\label{rescale CH theorem}
\end{theorem}

In this theorem, the parameter $\alpha$ can   diverge to infinity, which provides more flexibility  for the rescaled CH prior compared to the rescaled Mat\'ern prior. Although $\alpha$ is not a natural rescaling parameter as $\beta$ or $\phi$, its choice still affects the rate. Specifically, in this theorem  we   show that  when  $\alpha$ goes to infinity slowly, the  optimal minimax rate is obtained. The case where $\alpha$ goes to infinity quickly remains to be explored. 

In  Theorems \ref{rescale matern theorem} and \ref{rescale CH theorem}, when   $v=\eta$,  $\phi$ and $\beta$ are  fixed, i.e., the priors are non-rescaled,  we obtain the optimal minimax rate,  which can be expected since in this case the smoothness parameter of covariance function matches the the regularity $\eta$ of the ground truth. Mat\'ern process prior  with $\phi$ fixed is studied in  Theorem 5 of  \cite{van2011information}.

   The following theorem states that when $\beta$ goes to infinity too quickly, as in the rescaled Mat\'ern prior case, the posterior contraction rate is suboptimal. 
\begin{theorem}
Suppose we use a centered CH prior with  covariance function  (\ref{eq:CH}),  $\alpha>d/2+1$, $\alpha \leq C \sqrt{\ln \ln{n}}$  for sufficiently large $n$ and an arbitrarily large multiplicative constant $C$ that  does not depend on $n$, and $0<\beta<1$.   Then, for nonparametric regression with fixed design and additive Gaussian errors, we have,
$$
  \varphi_0(\varepsilon )\gtrsim \beta^{-d} \varepsilon^{-d/v}.
$$
When $v>\eta$ and  $\beta=o(n^{-\frac{v-\eta}{(2\eta+d)v}})$, the posterior contraction rate is suboptimal. Furthermore, when $\beta^{-d} \gtrsim n$, 
$$
 \Pi_n(w:\|w-w_0\|_n \leq 1 \mid Y_1,\ldots,Y_n)  \to 0,
$$
in probability $P_0^n$, i.e. the posterior does not contract.
\label{lower bound of rescaled CH}
\end{theorem}

\subsection{Adaptive Posterior Contraction Rates}\label{section3.4}
In the previous subsections, we obtained the optimal minimax rate by  choosing the rescaling parameter depending on the regularity of the function of interest, which is always unknown in practice.  \cite{van2009adaptive} consider a fully Bayesian alternative by putting a hyperprior on the rescaling parameter for the squared exponential process prior.  In this subsection, we follow the method of \cite{van2009adaptive}, and  obtain the optimal minimax rate
in a fully Bayesian setting simultaneously over a range of true regularity values,  for both Mat\'ern and CH processes.

Consider the Mat\'ern  case, where we put a prior on $\phi$, and   let $A=1/\phi$.  We  denote this hierarchical process by $W_M^A$. For the CH case, we put a prior on $\beta$. In this case, let $A=1/\beta$ and denote this  hierarchical process by $W_{CH}^A$. For simplicity, we abbreviate $W_M^A$ or $W^A_{CH}$ to $W^A$ by dropping the subscripts when we handle  either Mat\'ern  or CH  hierarchical processes. Now we assume that the distribution of $A$ possesses a Lebesgue density $g_A(\cdot)$ satisfying the condition: 
\begin{equation} \label{gamma prior tail }
C_1 a^p \exp \left(-D_1 a^{d} \right) \leq g_A(a) \leq C_2 a^p \exp \left(-D_2 a^{d}  \right) ,
\end{equation}
for positive constants $C_1, D_1, C_2, D_2$, non-negative constants $p$ and all  sufficiently large $a>0$. A gamma distribution on $A^{d}$ satisfies this condition.

Adaptive posterior rate can be obtained by verifying the following three conditions \citep{van2009adaptive} for Borel measurable subsets $B_n$ of $C(\mathcal{T})$ such that, for  sufficiently large $n$,
\begin{equation} \label{adaptive condition1}
P\left(\left\|W^A-w_0\right\|_{\infty} \leq \varepsilon_n\right)  \geq e^{-n \varepsilon_n^2},
\end{equation}
\begin{equation}\label{adaptive condition2}
P\left(W^A \notin B_n\right)  \leq e^{-4 n \varepsilon_n^2}, 
\end{equation}
\begin{equation}\label{adaptive condition3}
\log N\left({\varepsilon}_n, B_n,\|\cdot\|_{\infty}\right)  \leq n {\varepsilon}_n^2,
\end{equation}
where $\varepsilon_n$   is to be determined.  We prove the following result.
\begin{theorem}
    \label{adaptive matern}
Let $W$ be a centered  Gaussian process with  Mat\'ern  covariance function (\ref{eq:Matern}).  Put  a prior satisfying (\ref{gamma prior tail }) on random variable  $A=1/\phi$  and denote this hierarchical process by $W^A$. If $w_0 \in C^{\eta}(\mathcal{T} )\cap H^{\eta}(\mathcal{T})$ for some $\eta > 0$ and  $v \ge \eta$, then there exist Borel measurable subsets $B_n$ of  $C(\mathcal{T})$ such that conditions (\ref{adaptive condition1}), (\ref{adaptive condition2}) and (\ref{adaptive condition3}) hold, for  sufficiently large $n$, and $\varepsilon_n \asymp n^{-\eta /(2 \eta+d)}$.
\end{theorem}

By Theorem~\ref{adaptive matern} and an application of the proof of Theorem 3.3 of \cite{van2008rates}, one obtains  the following (optimal minimax) posterior contraction rate result for fixed design nonparametric  regression with hierarchical Mat\'ern  process priors.

\begin{theorem}
    Under the conditions of Theorem \ref{adaptive matern}, for
    fixed design nonparametric  regression with additive Gaussian error,      
$$
E_{w_0} \Pi_n(w:\|w-w_0\|_n >M n^{-\frac{\eta}{2\eta+d}} \mid Y_1,\ldots,Y_n)  \to 0,
$$
for any sufficiently large constant $M$, i.e.
    the posterior contracts at the (optimal minimax) rate $n^{-\frac{\eta}{2\eta+d}}$.

\end{theorem}

 Theorem 3.1 of \cite{van2009adaptive} can be seen to be closely connected to our Theorem \ref{adaptive matern}, since the exponential  process can be seen  as  a limiting case of Mat\'ern process when $v \to \infty$.

For  fixed design nonparametric  regression with hierarchical CH  process priors, similar to Mat\'ern case, we also have the following two theorems regarding the (optimal minimax) posterior contraction rate. We provide a proof for Theorem~\ref{adaptive CH}, while Theorem~\ref{th:vanCH} follows by an application of Theorem 3.3 of \citet{van2008rates} to the result of Theorem~\ref{adaptive CH}.
\begin{theorem}
    \label{adaptive CH}
Let $W$ be a centered  Gaussian process with  CH covariance function (\ref{eq:CH}).   Put a prior satisfying (\ref{gamma prior tail }) on random variable $A=1/\beta$ and denote this hierarchical process by $W^A$. If $\alpha>d/2+1$,  $w_0 \in C^{\eta}(\mathcal{T} )\cap H^{\eta}(\mathcal{T})$  for some $\eta > 0$ and  $v \ge \eta$,  then there exist Borel measurable subsets $B_n$ of  $C(\mathcal{T})$ such that conditions (\ref{adaptive condition1}), (\ref{adaptive condition2}) and (\ref{adaptive condition3}) hold, for  sufficiently large $n$, and $\varepsilon_n \asymp n^{-\eta /(2 \eta+d)}$.
\end{theorem}

\begin{theorem}\label{th:vanCH}
    Under the conditions of Theorem \ref{adaptive CH}, for
    fixed design nonparametric  regression with additive Gaussian error, 
$$
E_{w_0} \Pi_n(w:\|w-w_0\|_n >M n^{-\frac{\eta}{2\eta+d}} \mid Y_1,\ldots,Y_n)  \to 0,
$$
for any sufficiently large constant $M$, i.e.
    the posterior contracts at the (optimal minimax) rate $n^{-\frac{\eta}{2\eta+d}}$.
\end{theorem}

\section{Posterior Contraction Rates for Anisotropic  Covariance Functions}
\label{section4}
Under directional spatial effects, isotropy is no longer a realistic assumption for modeling. A similar argument can be made for other applications of multivariate random fields that warrant anisotropic modeling. Suppose the isotropic correlation function is $C(d(\bm{x},\bm{y}))$, where $d$ is Euclidean distance. Anisotropy can be introduced by applying $C(\cdot)$  to a non-Euclidean distance measure, obtained as Euclidean distance in a linearly transformed coordinate system. For the simple geometric anisotropy case \citep{haskard2007anisotropic,allard2016anisotropy},  consider a Mahalanobis-type distance: 
$$
\tilde{d}(\bm{x},\bm{y})=\sqrt{{(\bm{x}-\bm{y})}^TA{(\bm{x}-\bm{y})}},
$$
 where $A$ is a positive definite  matrix. When $A=\mathrm{diag}(a_i)$ is a diagonal matrix, covariance kernel $K(x,y)=C(\tilde{d}(\bm{x},\bm{y}))$ is termed the automatic relevance determination (ARD) kernel, and is widely used in the machine learning literature \citep{williams2006gaussian}.

Assume   $\bm{h}=\{h_1, \ldots,h_d\}$, where $h_i$s are scalars for $i=1,\ldots,d$. Let $\bm{B}$ be a positive definite $d \times d$ matrix with $ij$th entry  ${B}_{ij}$. We  define the anisotropic Mat\'ern covariance function  to be: 
$$
M(\bm{h};v,\bm{B},\sigma^2)=\sigma^2 \frac{2^{1-v}}{\Gamma(v)}\left(\sqrt{2v\left[\sum_{i=1}^{d}{{B}_{ij}}{h_ih_j}\right]}\right)^v {K}_v\left( \sqrt{2v\left[\sum_{i=1}^{d}{{B}_{ij}}{h_ih_j}\right]} \right),
$$
and the anisotropic CH covariance function to be:
$$
C(\bm{h};v,\alpha,\bm{B},\sigma^2)=\frac{\sigma^2 \Gamma(v+\alpha)}{\Gamma(v)}U\left(\alpha,1-v,v\left[\sum_{i=1}^{d}{{B}_{ij}}{h_ih_j}\right]\right).
$$
Then, the spectral density of the  anisotropic Mat\'ern covariance is: 
$$
m_M^{\bm{B}}(\bm{\lambda}) = \frac{\sigma^2(2 v)^{v}}{\pi^{d / 2}|\bm{B}|^{1/2}\left(2v+\bm{\lambda}^T\bm{B}^{-1}\bm{\lambda}  \right)^{v+d/2}} ;
$$
and the spectral density of the anisotropic CH covariance is: 
\begin{displaymath}
m^{\alpha,\bm{B}}_{CH}(\bm{\lambda}) =  \frac{\sigma^2 2^{v-\alpha} v^v }{\pi^{d/2}\Gamma(\alpha)|\bm{B}|^{1/2}}\int_0^{\infty} (2v\phi^{-2}+\bm{\lambda}^T\bm{B}^{-1}\bm{\lambda})^{-v-\frac{d}{2}}  \phi^{-2(v+\alpha+1)} \exp{(-\frac{1}{2\phi^2})d\phi^2}.
\end{displaymath}
The spectral densities of the anisotropic Mat\'ern  and CH covariances can be obtained by applying Fourier transform to the covariance functions and  using variable transformation $\bm{h}=\bm{B}^{-1/2}\bm{t}$. Then we can deal with this Fourier transform like the isotropic case.  Here we call $\bm{B}$ an anisotropy matrix. Suppose $\lambda_{\min}, \lambda_{\max}$ are the smallest and largest  eigenvalues of the anisotropy matrix $\bm{B}$. If we impose some restriction on the eigenvalues of $\bm{B}$, then we have the following posterior contraction rate results for stationary Gaussian process priors with anisotropic Mat\'ern and CH covariance functions.
\begin{theorem}\label{rescale anisotropic theorem}
Assume $w_0 \in C^{\eta}(\mathcal{T}) \cap H^{\eta}(\mathcal{T})$, $v \ge \eta>0$ and $\lambda_{\min}/\lambda_{\max} \ge C >0$, where $C$ is a constant.  We use a centered  stationary Gaussian process prior with  anisotropic Mat\'ern covariance function $M(h;v,\bm{B},\sigma^2)$ or with  anisotropic CH covariance function  $C(\bm{h};v,\alpha,\bm{B},\sigma^2)$,  $\alpha>d/2+1$,$\alpha \leq C_0 \sqrt{\ln \ln{n}}$  for sufficiently large $n$ and an arbitrarily large multiplicative constant $C_0$ that  does not depend on $n$ and $\lambda_{\max}=n^{\frac{v-\eta}{(2\eta+d)v}}$. Then,  for nonparametric regression with fixed design and additive Gaussian errors, 
$$
E_{w_0} \int \|w-w_0\|_n^2 d \Pi_n(w\mid Y_1,\ldots,Y_n) \lesssim (n^{-\frac{\eta}{2\eta+d}})^2,
$$
i.e. the posterior contracts at the (optimal minimax) rate $n^{-\frac{\eta}{2\eta+d}}$. 

\end{theorem}

 In Theorem \ref{rescale anisotropic theorem}, the condition $\lambda_{\min}/\lambda_{\max} \ge C >0$ implies the non-Euclidean distance is approximately the Euclidean distance times a constant. This theorem shows that under these conditions,  Gaussian process prior with anisotropic  covariance
function  and Gaussian process prior with isotropic  covariance
function yield similar posterior concentration properties.

We also mention that \cite{bhattacharya2014anisotropic} discuss  Gaussian process priors with anisotropic  covariance functions. Their anisotropy matrix $\bm{B}$ is the diagonal matrix, with gamma prior  on the diagonal elements (after taking some powers).  Their Bayesian procedure leads to the minimax optimal
rate of posterior contraction (up to a logarithmic factor) for the anisotropic H\"older space they define. They also  prove that the optimal prior choice in the isotropic case leads to a sub-optimal convergence rate if the true function has anisotropic smoothness. In contrast to their work, we do not need to assume the anisotropy matrix $\bm{B}$ is the diagonal matrix. This is relevant because even in the simple geometric anisotropy case, $\bm{B}$ need not be diagonal \citep{haskard2007anisotropic}. Further, we establish the optimal minimax contraction rate without the logarithmic factor.  However, the performance of our anisotropic prior on their anisotropic H\"older space is unknown.

\section{Simulation Results} \label{section5}
In this section, we consider nonparametric regression with fixed design and additive Gaussian errors as described in Section \ref{section2.4}.  We  estimate the regression function $w$ based on observations $Y_1, \ldots, Y_n$  and  fixed covariates $\bm{x}_1, \ldots, \bm{x}_n$ from the set $\mathcal{T}=[0,1]^d$.  We also assume $\varepsilon_i \stackrel{\text { i.i.d }}{\sim} N\left(0, \omega\right)$, with $\omega$ known.  

In what follows, we compare  for  different true regression functions the posterior concentration performances of  rescaled  Mat\'ern and  CH priors with our choice of rescaling, and rescaled squared exponential process priors with rescaling parameter set as in \cite{van2007bayesian}. We also consider   Mat\'ern, CH, squared exponential priors with parameters estimated by the method of  maximum likelihood;  and hierarchical Mat\'ern,  CH, and  squared exponential processes as in Section  \ref{section3.4},  by putting suitable priors on  the rescaling parameters.  We take the squared exponential covariance function to be $\exp{(-\|h\|^2/c)}$.   We use the procedures of \cite{stein1999interpolation} and \citet{ma2022beyond} to compute the MLE estimators for all parameters, except $v$, under the chosen covariance models. However, we fix the smoothness parameter $v$ and do not estimate it, since there are known identifiability issues with estimating the smoothness parameter \citep{gu2018robust}. 

According to the proof of theorems in Section \ref{section3},  we obtain the minimax rate  when   $f \in H^{\eta}[0,1]$ but $f \notin H^{\theta}[0,1]$ for all $\theta>\eta$.  Further, if we let the true function to be analytic, the posterior  contraction rate is the parametric rate $n^{-1/2}$.  Thus, for a reasonable choice of the true function where a difference under various covariances can be expected, we prefer rough functions. For example, when $d=1$, the realization of  Brownian motion is continuous but 
nowhere differentiable almost surely \citep{morters2010brownian}. Actually, Brownian motion is almost everywhere locally
$\alpha$-H\"older continuous for all $\alpha<1/2$; and for all $\alpha>1/2$,   it fails to be locally $\alpha$-H\"older continuous almost everywhere.  Section 4 of  \cite{kanagawa2018gaussian}  demonstrates that   GP sample functions are \emph{rougher}, or less regular, than RKHS (corresponding to GP prior) functions,  so  taking  realizations of GP  as true functions  is appropriate in our simulations. 

We first analyze the simple case when the dimension $d$ of the covariate $\bm{x}$ is $1$. The performance of posterior concentration is illustrated by comparing    the  predictive performance  based on  mean-squared prediction errors (MSPE), empirical coverage of the 95\% predictive confidence intervals (CVG) and the average length of the 95\% predictive confidence intervals (ALCI) at held-out locations. The $d=2$ case is explored similarly to the $d=1$ case.

\subsection{The  Case of $d=1$}
We simulate $n=300$ data points sampled uniformly from the interval $[0,1]$.  Among these, $100$ data points  are picked uniformly as the testing set, the remaining $200$ data points constitute the training set. Parameters are estimated by the method of maximum likelihood or a Bayesian approach; or they are set by rescaling, as described in the next paragraph. Using the set of the estimated parameters, we then  calculate MSPEs, CVGs and ALCIs on other 30 replicates of the training and testing sets of the same size.  

We  set the smoothness parameter $v=2$  for both CH and Mat\'ern. For the rescaling approach, we estimate  the parameters  $\alpha,\beta,\phi,c$ by the method of maximum likelihood, before rescaling these parameters in order to make CVG as close to 0.95 as possible on our testing and training data set. We call these parameters as the optimal parameter choice for rescaling.  For  hierarchical Mat\'ern,  CH, and  squared exponential processes,  we generate the posterior samples of the rescaling parameters given data  by the Metropolis-Hastings algorithm \citep{chib1995understanding} with 500 burn-in samples, followed by {5000} MCMC samples. {To satisfy the condition in Theorem \ref{adaptive matern},  we put $\mathrm{Gamma}(1,1)$ priors on $A^{kd}$ for  hierarchical Mat\'ern,  CH processes ($k$ to be determined for each specific case) and   put $\mathrm{Gamma}(1,1)$ priors on $A^{d}$ for  hierarchical  squared exponential process \citep{van2009adaptive}.}

We first let $w(x)$ be a realization of Brownian motion (times $100$). Then $w(x)$ has regularity $1/2$, and let the noise variance be $\omega=1$.  In Figure~\ref{d=1 brownian cv},  we   compare  the MSPE, ALCI and CVG  of rescaled, hierarchical and MLE-based CH, Mat\'ern, squared exponential process priors.  
For the hierarchical model, we set the smoothness parameter in CH and Mat\'ern  covariance to be 5 and  $k=3$  to satisfy the condition in Theorem \ref{adaptive matern}. 

We find the performance of hierarchical method with same type of covariance functions  to be better than the rescaling and MLE methods. The hierarchical method has  better CVGs, MSPEs and ALCIs.  Comparing the rescaling  and MLE methods, it is apparent that the rescaling method has better CVGs while  ALCIs are larger. MSPEs of rescaling method are also slightly better.  CH process prior outperforms Mat\'ern prior with smaller ALCIs and MSPEs, and the   Mat\'ern prior  also outperforms squared exponential prior in this case.  This can be expected because rescaled CH and Mat\'ern priors can attain the optimal minimax rate for  $\eta$-regular functions, while
the rescaled squared exponential prior can only achieve the   minimax rate for  $\eta$-regular function up to a logarithmic factor \citep{van2007bayesian}.  These simulations also indicate that CH   processes are more suitable for rough true function than Mat\'ern, and both of these are better than the smooth squared exponential process prior.

\begin{figure}[!t]
    \centering
    {\includegraphics[width=1\textwidth]{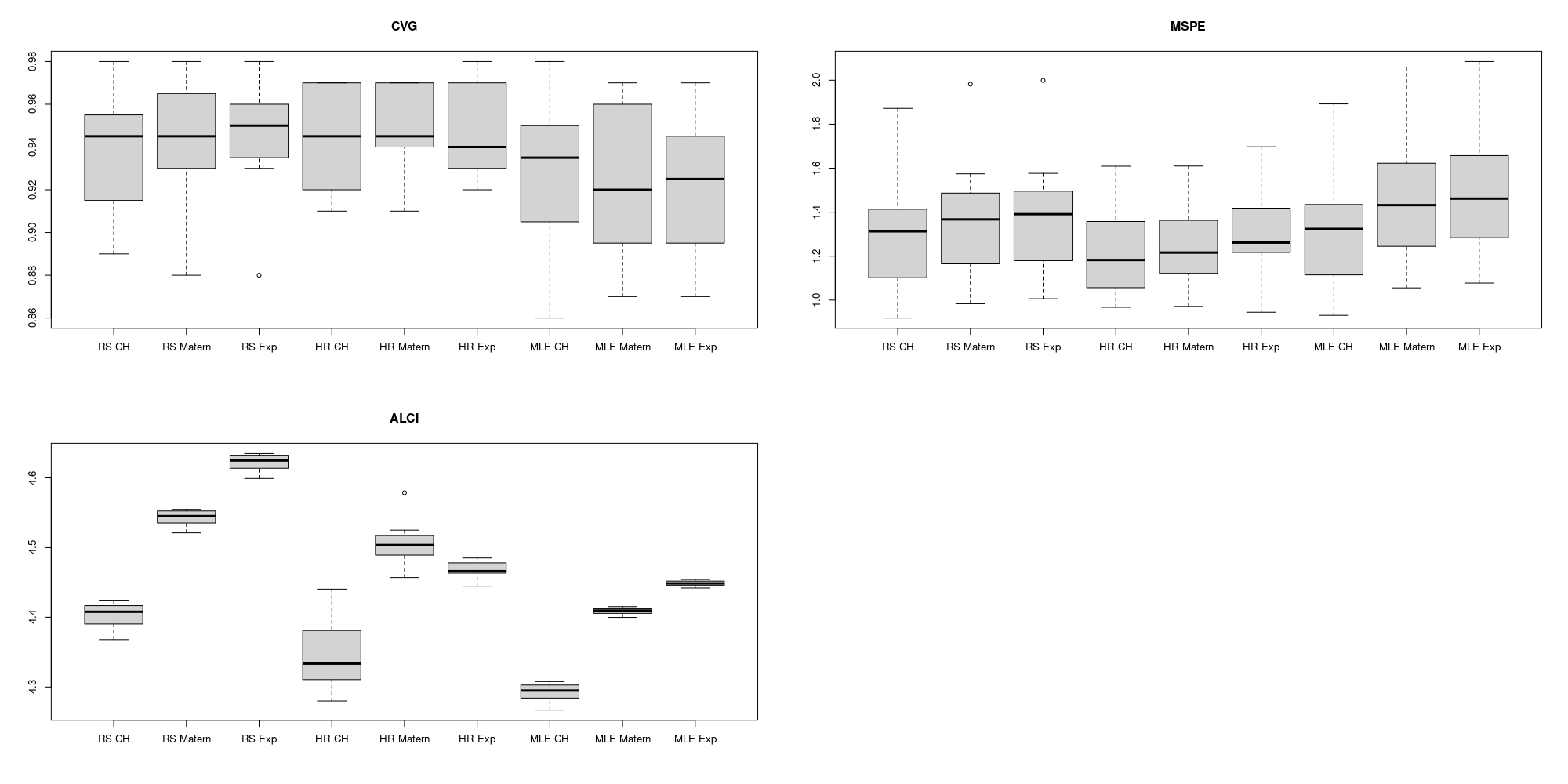}} 
    \caption{ (Left to right). Boxplots of coverage of 95\% confidence intervals (CVG), mean squared prediction error (MSPE) and average length of the confidence intervals (ALCI). Results are for CH, Mat\'ern and squared Exponential (Exp) covariances, with parameters set via rescaling (RS), hierarchical (HR) or MLE methods. Boxplots are computed over 30 randomly chosen training and testing data sets. Here $d=1$ and the true function $f(x)$ is a realization of Brownian motion.\label{d=1 brownian cv}}
\end{figure}

\begin{figure}[!b] 
    \centering
    {\includegraphics[width=1\textwidth]{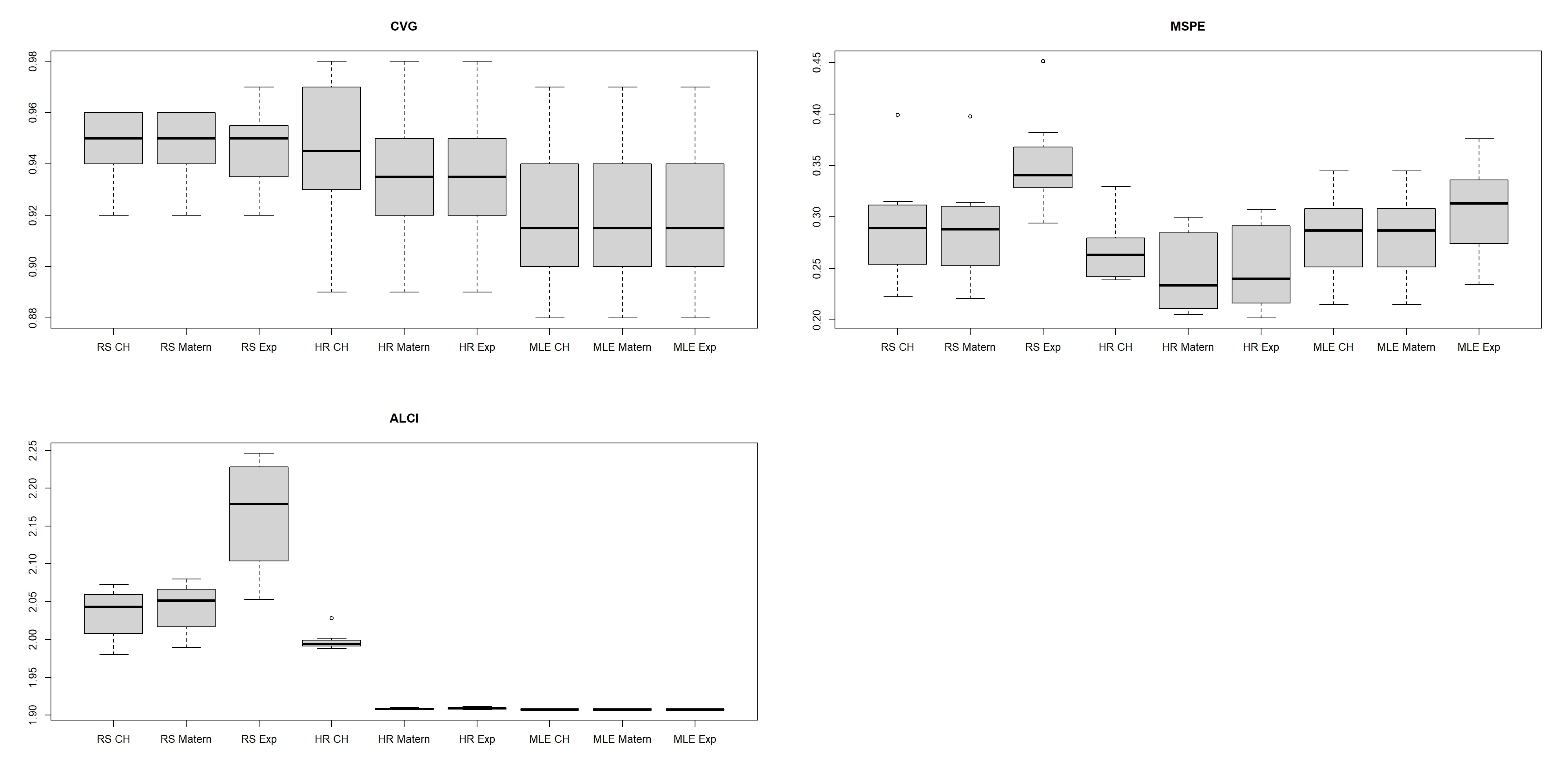}} 
    \caption{(Left to right). Boxplots of coverage of 95\% confidence intervals (CVG), mean squared prediction error (MSPE) and average length of the confidence intervals (ALCI).  Results are for CH, Mat\'ern and squared Exponential (Exp) covariances, with parameters set via rescaling (RS), hierarchical (HR) or MLE. Boxplots are computed over 30 randomly chosen training and testing data sets. Here $d=1$, and the true function $f(x)$ is a realization of stationary Gaussian process with mean $0$ and covariance function $M(h;1,1,1)$. \label{d=1 m111 cv}}
\end{figure}
Next,  we consider $w(x)$ to be  a realization of a stationary Gaussian process with mean 0 and covariance function $M(h;1,1,1)$, and set $\omega=0.5$. By \cite{porcu2023mat}, the RKHS of the corresponding Gaussian process is the Sobolev space $H^{3/2}[0,1]$.  Section 4 of \cite{kanagawa2018gaussian} shows the sample path of this process does not belong to the RKHS almost surely. However Corollary 1 of \cite{scheuerer2010regularity} also confirms this sample path is almost surely in $H^1[0,1] $, so $w(x)$ is smoother than the realizations of Brownian motion, but still has regularity less than $1.5$.  For the hierarchical model, we set the smoothness parameter in CH and Mat\'ern  covariance to be 5 and  $k=3$  to satisfy the condition in Theorem \ref{adaptive matern}.   Figure \ref{d=1 m111 cv} deals with this example and displays the same quantities as  Figure \ref{d=1 brownian cv}.  In this case 
the performance of rescaling and hierarchical methods with same type of covariance functions  are better than  the MLE method with  better CVGs, MSPEs and ALCIs.  The hierarchical method has smaller MSPEs than the rescaling method,  but its CVG is slightly worse. Compared to the Brownian motion example, the true function is smoother and  the difference of performance between CH and Mat\'ern priors is negligible in this case.  Both of these still outperform the squared exponential  prior.

We also notice in our simulations that  when $\beta$, $\phi$, $c$ go to $0$ too quickly as $n \to \infty$,  posterior concentration results do not hold for these three priors. This is supported by  Theorems \ref{lower bound of rescaled matern} and \ref{lower bound of rescaled CH}.

\subsection{The  Case of $d=2$}
In the $d=2$ case,  we simulate on  $n = 300$ data points uniformly drawn from $[0,1]^2$ and select 80 data points uniformly as the testing data, the rest as training data. We  set the smoothness parameters in CH and Mat\'ern Class to be $v=2$. For this data, the  parameters of interest are obtained by MLE and rescaled methods.  For the hierarchical model, we set the smoothness parameter in CH and Mat\'ern  covariance to be 6 and  $k=7$  to satisfy the condition in Theorem \ref{adaptive matern}.  We repeat the procedure on 30 randomly picked testing and training data sets. 

We  let the  true function $f(x)$ be a realization of stationary Gaussian process with mean 0 and covariance function $M(h;1,1,1)$, and $\omega=1$. This true function is differentiable but not in the Sobolev space $H^2([0,1]^2)$. From Figure~\ref{d=2 m111 cv},  among the 3 methods, MLE has worst  CVGs and ALCIs,  while the performances of rescaling and  hierarchical methods are   similar. Both CH and Mat\'ern priors outperform squared exponential process prior, and CH priors are slightly better than Mat\'ern.

\begin{figure}[!h]
    \centering
    {\includegraphics[width=1\textwidth]{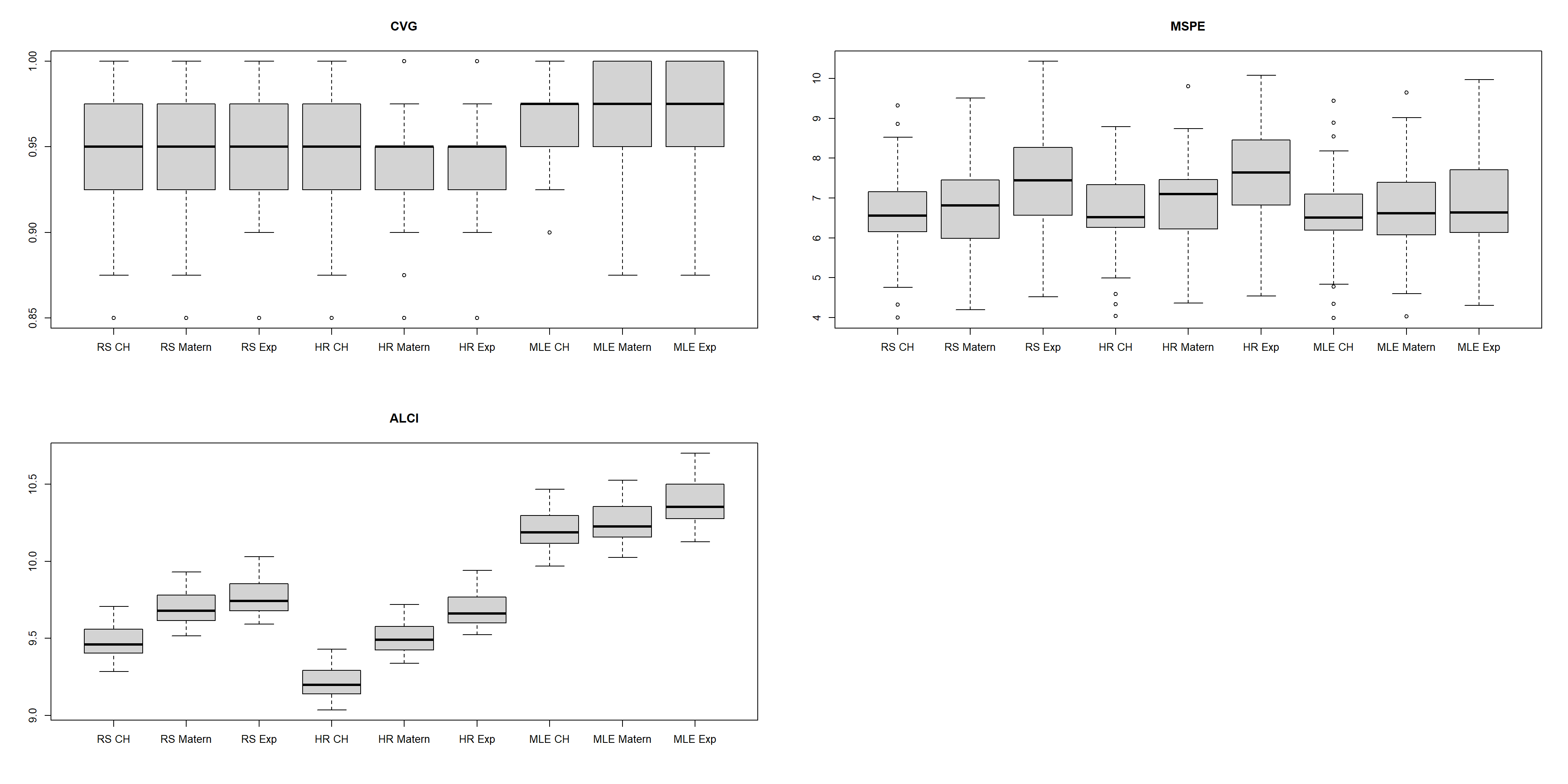}} 
    \caption{(Left to right). Boxplots of coverage of 95\% confidence intervals (CVG), mean squared prediction error (MSPE) and average length of the confidence intervals (ALCI). Results are for CH, Mat\'ern and squared Exponential (Exp) covariances, with parameters set via rescaling (RS), hierarchical (HR) or MLE. Boxplots are computed over 30 randomly chosen training and testing data sets. Here $d=2$, and the true function $f(x)$ is a realization of stationary Gaussian process with mean 0 and covariance function $M(h;1,1,1)$.\label{d=2 m111 cv}}
\end{figure}

\section{Results on Atmospheric NO$_2$ data} \label{section6}
In this section we study the relationship between location and  the level of Nitrogen Dioxide (NO$_2$), a known environmental pollutant,  with  the nonparametric normal regression model described in Section \ref{section2.4}. Our data are the levels of NO$_2$ concentration,  measured in parts per million (ppm), the city of York, UK from December, 2022. We aim to predict the level of NO$_2$ by location information $(X,Y)$ = (latitude, longitude).  To evaluate the performance,  we randomly select 65 data points as the validation set and the rest 154 data points as the training set.  The training and validation data sets are displayed in Figure \ref{training and validation}.  We select parameters based on the training set, and then evaluate the prediction performance of our nonparametric normal regression model with rescaled and hierarchical CH, Mat\'ern and squared exponential priors on the testing set.

In our theoretical results,  the smoothness parameter $v$ should be greater than the regularity of the true function. Therefore, here we set $v=5$ for CH and Mat\'ern covariances as a sufficiently large $v$.  Then, we use maximum likelihood method to estimate the  parameters in CH, Mat\'ern  and squared exponential covariance functions. We set those estimated parameter as initial value and we rescale $\beta$, $\phi$, $c$ to  make CVGs  be as near 0.95 as possible.   For the  hierarchical model, we set the smoothness parameter in CH and Mat\'ern  covariance to be 10 and  $k=4$  to satisfy the condition in Theorem \ref{adaptive matern}. To avoid singularity in matrix calculation, we center and scale $X$ with $100(X-\bar{X})$ and $Y$ with $100(Y-\bar{Y})$. We also rescale  the level of NO$_2$ concentration by dividing it by the sample maximum.  The results are repeated over 30 random splits of the data set, into training and testing sets of the same size. We summarize the results in Figure~\ref{real data cv}. From  the boxplots, we observe the rescaled method has better CVG than the MLE and hierarchical method. However, its MSPEs are worse.  The MLE and hierarchical methods have very similar performances.  In this case, all methods related to the CH process prior 
perform much better than the Mat\'ern and squared exponential process priors.

Figure \ref{figure.sd}  displays the scatterplots of residuals versus predicted values under rescaled and hierarchical CH, Mat\'ern and squared exponential methods on the validation set, along with the posterior predictive intervals. Out of the 65 validation data points, 63, 61, 61 of the validation data  points lie inside the 95\% predictive intervals for the rescaled CH, Mat\'ern and squared 
 exponential method;  63, 64, 64 of the validation data  points lie inside the 95\% predictive intervals for the hierarchical CH, Mat\'ern and squared 
 exponential method. These methods have similar coverage.
However, the 95\% predictive intervals from the rescaled and  hierarchical CH method are  shorter in general compared to rescaled and hierarchical Mat\'ern or squared exponential. Overall, rescaled and hierarchical CH perform the best, with rescaled and hierarchical Mat\'ern, rescaled squared exponential performing similarly, and both performing worse than CH.

\begin{figure}[!t]
\centering
\includegraphics[width=12cm]{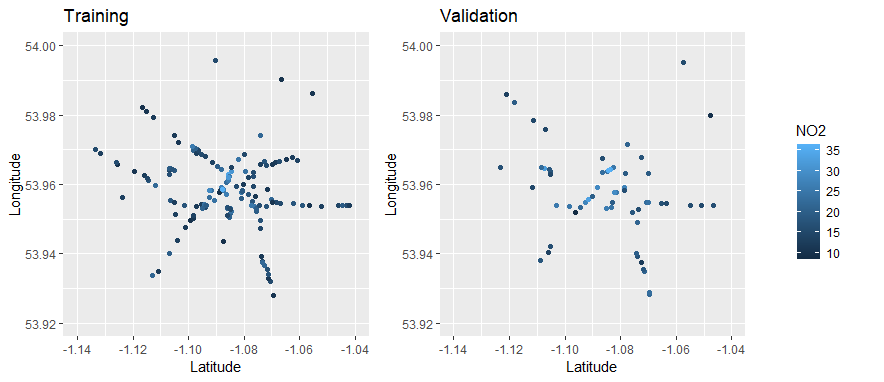}
\caption{Scatter plot of NO$_2$ measurements in York, UK in December 2022.}
    \label{training and validation}
\end{figure}

\begin{figure}[!h] 
    \centering
    \subfigure{\includegraphics[width=1\textwidth]{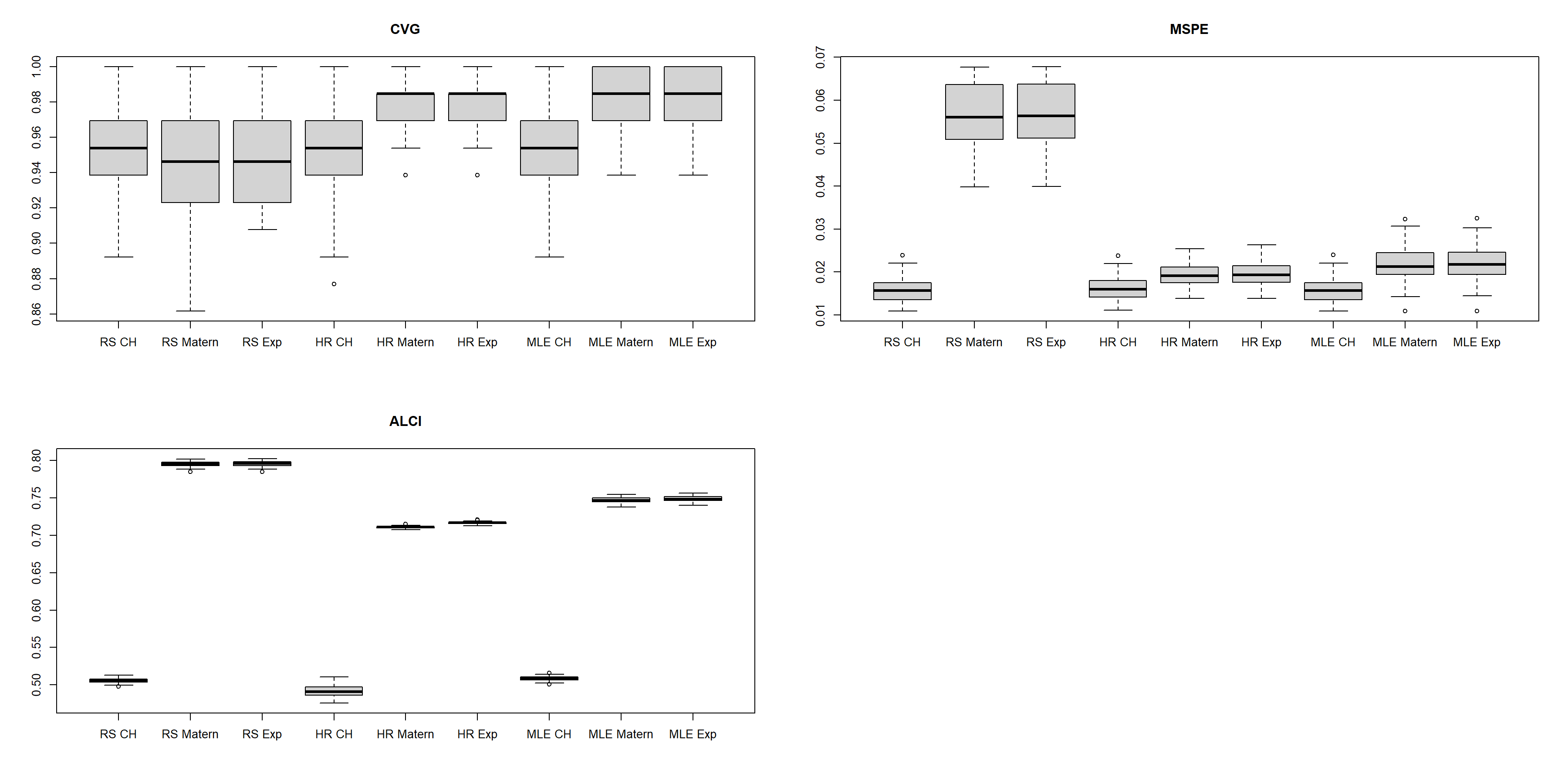}} 
    \caption{(Left to right). Boxplots of coverage of 95\% confidence intervals (CVG), mean squared prediction error (MSPE) and average length of the confidence intervals (ALCI). Results are for CH, Mat\'ern and squared Exponential (Exp) covariances, with parameters set via rescaling (RS), hierarchical (HR) or MLE, for the NO$_2$ data.\label{real data cv}}
\end{figure}

\begin{figure}[!h]
    \centering
    \subfigure{\includegraphics[width=1\textwidth]{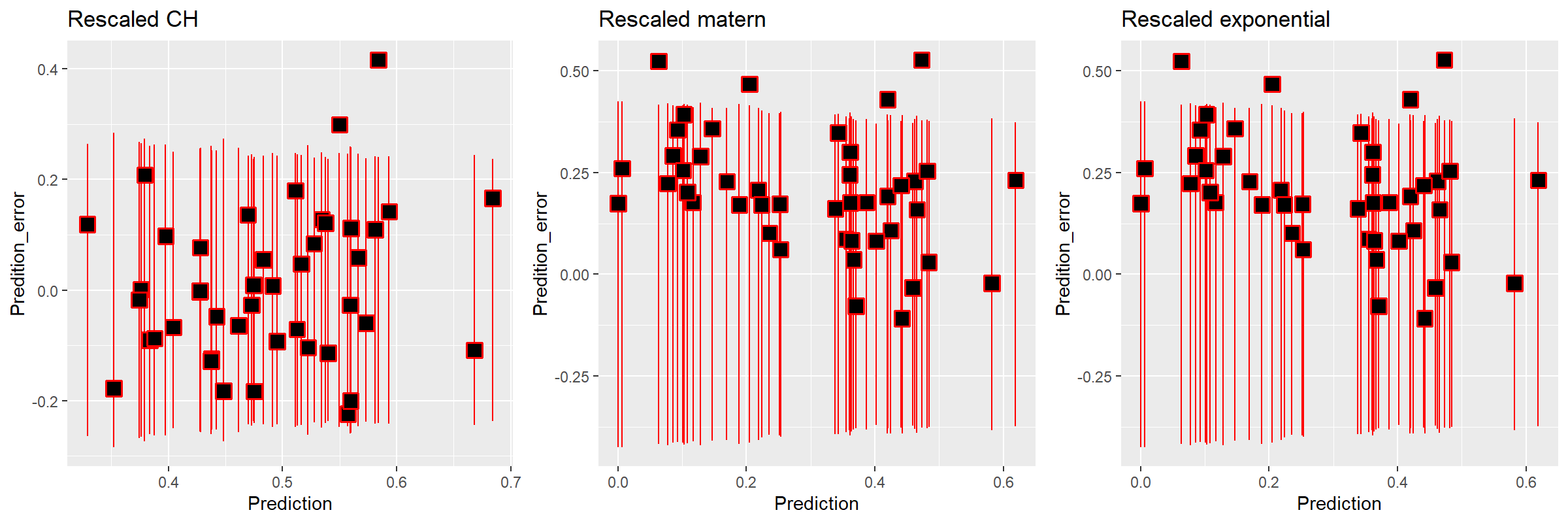}} 
    \subfigure{\includegraphics[width=1\textwidth]{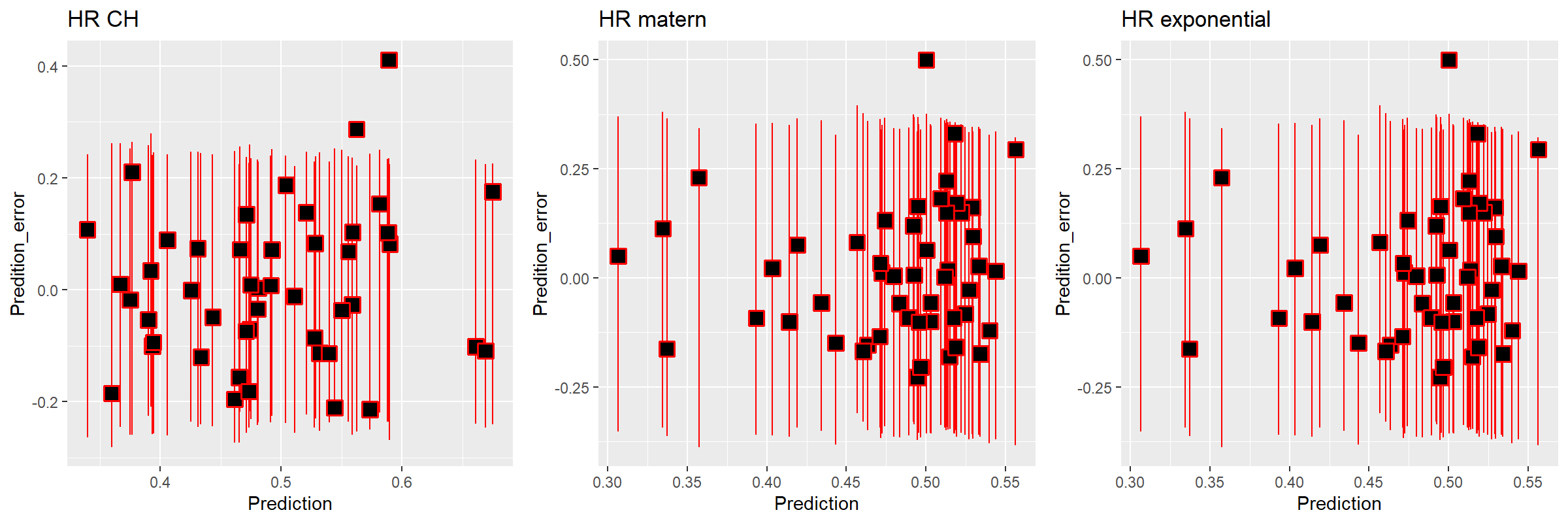}}
   
    \caption{Prediction error vs. residuals for the NO$_2$ data validation set for CH, Mat\'ern and squared exponential covariances under rescaling and hierarchical settings. Bars indicate posterior predictive 95\% intervals.}
        \label{figure.sd}
\end{figure}

\section{Discussion} \label{section7}
This paper studies  posterior concentration properties of nonparametric normal regression with fixed design. For $\eta$-regular true functions,  we find that by rescaling the parameters in Mat\'ern and CH classes, we can obtain  optimal minimax  posterior contraction rate.  We also obtain the optimal minimax  posterior contraction rate for hierarchical  Mat\'ern and CH process priors without a knowledge of the true regularity, resulting in a practically useful procedure. Although we demonstrate the optimal minimax rates, there are still areas of further investigations.

First of all, we obtain the optimal minimax posterior contraction rate by rescaling. However,   the choice of the rescaling parameter depends on the smoothness of the function of interest ($\eta$ in our case), which is always unknown in practice. We handle this problem by assigning  a hyperprior on  rescaling parameter.  It is also possible to  choose the lengthscale in a data-dependent manner.  \cite{szabo2013empirical} apply an empirical Bayes method and obtain  the rescaling parameter by  maximizing the marginal likelihood. Similar ideas are discussed  in \cite{knapik2016bayes} and \citet{rousseau2017asymptotic}.
The posterior contraction rate for nonparametric regression model with stationary Gaussian process priors remains to be explored under a lengthscale parameter set by maximizing the marginal likelihood in an empirical Bayes procedure. { \cite{castillo2024deep} generalize the approach of \cite{van2009adaptive} and introduce deep horseshoe Gaussian process as prior, showing this prior leads to near minimax-optimal contraction rates  for their compositional function classes. Following our approach, one may also study how rescaled and hierarchical  Mat\'ern or CH process priors perform on these function classes.}

Our theoretical results  only deal with fixed design over  compact domains.  Following  the results of metric entropy  for function spaces on unbounded domains as in  \cite{nickl2007bracketing}, a study of posterior contraction over unbounded domains is an interesting avenue for  future work.

In our paper, we restrict our interest to nonparametric normal regression with fixed design.  For the random design case,  one may assume that given the function $f:[0,1]^d \rightarrow \mathbb{R}$ on the $d$-dimensional unit cube $[0,1]^d$,  the data $\left(X_1, Y_1\right), \ldots,\left(X_n, Y_n\right)$ are independent, $X_i$ having a density $\pi$ on $[0,1]^d$, and $Y_j$s are generated according to $Y_j=f\left(X_j\right)+\varepsilon_j$,  errors $\varepsilon_j \sim N\left(0, \sigma^2\right)$  are independent given $X_j$. \cite{van2011information,pati2015optimal,10.1214/20-AOS2043} obtain  posterior contraction rates for random design case, and their works show the key to get (nearly) optimal rate is to define a proper discrepancy measure.  In a future work, one may attempt to construct
the discrepancy measure under which rescaled Mat\'ern and CH  process priors attain the optimal minimax rates under random design.
\begin{appendix}
\label{appendix}

\section{Proofs of Main Results}
\subsection{Proof of Lemma \ref{matern small ball}} \label{proof matern small ball}
\begin{proof}
The small ball exponent can be obtained from the metric entropy (logarithm of the $\epsilon$-covering number) of  unit ball $\mathbb{H}_1$ of the RKHS of the Gaussian process $W$  \citep{li1999approximation}. The transform $\mathcal{F}_{\phi} \psi$  of $\psi$  given in (\ref{rescaled matern map}) is, up to constants, the function $g= \psi \cdot {m_{M}^{\phi}}$, and for the minimal choice of $\psi$ as in (\ref{RKHS of rescaled matern}), for Mat\'ern covariance we have:
\begin{displaymath}
\|\mathcal{F}_{\phi} \psi\|_{\mathbb{H}^{\phi}}^2  =\int |g(\lambda)|^2 \left({m_{M}^{\phi}}(\lambda)\right)^{-1} d\lambda\\
=\int |g(\lambda)|^2 (1+\lambda^2)^{(v+d/2)}\frac{(1+\lambda^2)^{-(v+d/2)}}{m_{M}^{\phi}(\lambda)} d\lambda.
\end{displaymath}
Since $\frac{(1+\lambda^2)^{-(v+d/2)}}{m_{M}^{\phi}(\lambda)} \gtrsim \phi^{ 2v}$, we have $\|\mathcal{F}_{\phi} \psi\|_{\mathbb{H}^{\phi}}^2  \geq  C \phi^{2v} \|g\|^2_{2,2,v+d/2}$, and thus the unit ball of the RKHS is contained in the Sobolev ball with radius $\phi^{-v}$ (up to a constant) of order $v+d/2$.  By  Theorem 2.7.4 in \cite{wellner2013weak}, the metric entropy of   such a Sobolev ball is bounded  above by a constant times $(\phi^{-v} )^{\frac{d}{v+d/2}}  \varepsilon^{-\frac{d}{v+d/2}}$. Next, by Theorem 1.2 of \cite{li1999approximation},
\begin{equation} \label{small ball of matern invariation}
    \varphi_0(\varepsilon) \lesssim \varepsilon ^{-\frac{2 \frac{d}{v+d/2}}{2-\frac{d}{v+d/2}}}\left[(\phi^{-v} )^{\frac{d}{v+d/2}} \right]^{\frac{2v+d}{2v}}=\varepsilon^{-d/v} \phi^{-d}.
\end{equation}
From the proof of Proposition 3.1 of \citet{li1999approximation}, this bound holds for all $\varepsilon >0$ satisfying:
\begin{displaymath}
    \phi^{\frac{vd}{v+d/2}} \lesssim (\varphi_0(\varepsilon/2))^{\frac{d}{2(v+d/2)}} \varepsilon^{-\frac{d}{v+d/2}}.
\end{displaymath}
 By assumption  we also have $\phi<1$. Thus, 
 $$
\varphi_0(\varepsilon/2) = - \log P\left(\sup_{t \in \mathcal{T}} |W_t^\phi| \le \varepsilon/2\right)  \ge - \log P\left(\sup_{ t \in \mathcal{T}} |W_t | \le \varepsilon/2\right),
$$
where the last inequality follows directly form the definition of the  rescaled process and 
 (\ref{small ball of matern invariation}) holds for all
 $\varepsilon$ in an interval independent of $\phi$, since the right hand side is independent of $\phi$. This  completes the proof.
\end{proof}

\subsection{Proof of Lemma \ref{matern decentering}} \label{proof matern decentering}
\begin{proof}
Let $\zeta$, $\zeta_{\phi}$, $h$ be the same construction as in the proof of Lemma 11.37 in \cite{ghosal2017fundamentals}. Let $\kappa: \mathbb{R} \rightarrow \mathbb{R}$ be a function with a real, symmetric Fourier transform $\hat{\kappa}(\lambda)=(2 \pi)^{-1}\int e^{i \lambda t} \kappa(t) d t$, then $\hat{\kappa}$ equals $1/(2 \pi)$ in a neighborhood of 0 which has compact support, with  $\int \kappa(t) d t=1$ and $\int(i t)^k \kappa(t) d t=0$ for $k \geq 1$. For $t=\left(t_1, \ldots, t_d\right)$, define $\zeta(t)=\prod_{i=1}^{d}\kappa\left(t_i\right)$. Then $\zeta(t)$ integrates to 1, has finite absolute moments of all orders, and vanishing moments of all orders bigger than 0.

For $\phi>0$, set $\zeta_\phi(x)=\phi^{-d} \zeta(x / \phi)$ and $h=\zeta_{\phi^{\theta}}* w_0$, where $\theta \ge 1$ is to be determined. By similar arguments in \cite{van2009adaptive}, it follows that $\left\|w_0-\zeta_{\phi^\theta} * w_0\right\|_{\infty} \leq C_{w_0} \phi^{\eta \theta}$ and $C_{w_0}$ only depends on $w_0$. We assume that the support of $\hat{\zeta}(\lambda)$ is in the set $\{\lambda:\|\lambda\| \le M\}$. The Fourier transform of $h$ is  $\hat{h}(\lambda)=\hat{\zeta}(\phi^{\theta} \lambda)\hat{w_0}(\lambda)$. Then $h  =  2 \pi \int e^{-it\lambda} \hat{\zeta}(\phi^{\theta} \lambda)\hat{w_0}(\lambda) d \lambda = 2 \pi \mathcal{F}_{\phi}\left( \frac{\hat{\zeta}(\phi^{\theta} \lambda)\hat{w_0}(\lambda)}{m_M^{\phi}(\lambda)}\right)$.
By (\ref{RKHS of rescaled matern})  we have: 

\begin{equation}
\begin{split}
\|h\|_{\mathbb{H}^{\phi}}^2  & =  \left\| 2 \pi \mathcal{F}_{\phi}\left( \frac{\hat{\zeta}(\phi^{\theta} \lambda)\hat{w_0}(\lambda)}{m_M^{\phi}(\lambda)}\right)\right\|_{\mathbb{H}^{\phi}}^2 \\
& \leq (2\pi)^2  \int|\hat{\zeta}(\phi^{\theta} \lambda)\hat{w_0}(\lambda)|^2 \frac{1}{m^{\phi}_M(\lambda)} d\lambda \\ 
&\leq \tilde{D}_{w_0} \cdot  \sup_\lambda\left[(1+{\|\lambda\|^2})^{-\eta} \left(m^{\phi}_M(\lambda)\right)^{-1} |\hat{\zeta}(\phi^{\theta} \lambda)|^2\right] \times \|w_0\|^2_{2,2,\eta}\\
& =  \tilde{D}_{w_0} \cdot \sup_{\|\lambda\| \leq M/\phi^{\theta}}\left[(1+{\|\lambda\|^2})^{-\eta}\left(m^{\phi}_M(\lambda)\right)^{-1} |\hat{\zeta}(\phi^{\theta}\lambda)|^2\right] \times \|w_0\|^2_{2,2,\eta}\\
&\leq D_{w_0} \cdot \sup_{\|\lambda\| \leq M/\phi^{\theta}}\left[(1+{\|\lambda\|^2})^{-\eta}\left(m^{\phi}_M(\lambda)\right)^{-1} \right] \times \|w_0\|^2_{2,2,\eta}\\
&= D_{w_0} \cdot\max_{\|\lambda\|=0,M/\phi^{\theta}}\left[(1+{\|\lambda\|^2})^{-\eta}\left(m^{\phi}_M(\lambda)\right)^{-1} \right] \times \|w_0\|^2_{2,2,\eta} \\
&= D_{w_0} \cdot \max \left\{\phi^{-d}, \phi^{2v-2\theta(v+d/2-\eta)}\right\}\times \|w_0\|^2_{2,2,\eta},
\end{split}
\end{equation}
where $\tilde{D}_{w_0}$, $D_{w_0}$ only depend on $w_0$, and the  second last equality is due to the fact that $\log[(1+{\|\lambda\|^2})^{-\eta}\left(m^{\phi}_M(\lambda)\right)^{-1} ]$  attains its maximum at the boundary, i.e., $\|\lambda\|=0$ or $M/\phi^{\theta}$ (by taking derivative with respect to $\|\lambda\|^2$ ).  When $\theta > \frac{2v+d}{2v+d-2 \eta}$, we have $\|h\|_{\mathbb{H}^{\phi}}^2 \lesssim  \phi^{2v-2\theta(v+d/2-\eta)}$.
\end{proof}

\subsection{Proof of Theorem \ref{lower bound of rescaled matern}}
\begin{proof}
For the rescaled Mat\'ern class,  when $\|\lambda\| \ge \phi^{-1}$, its spectral density satisfies:
$$
m^{\phi}_M(\lambda) \ge c_0 \phi^{-2v} \|\lambda\|^{-(2v+d)},
$$
where $c_0$ does not depend on $\phi$.
By  the corollary in  \cite{lifshits1987small}, we have:
$$
P\left(\sup_{t \in \mathcal{T}} |W_t^\phi| \le \varepsilon\right) \le \exp(-C \phi^{-d} \varepsilon^{-d/v}),
$$
where $C$ is a constant that only depends on $v$ and $d$. Thus, we have $\varphi_0(\varepsilon) \gtrsim \phi^{-d} \varepsilon^{-d/v}$.

The second part of this theorem can be obtained by applying Theorem 11.23 of 
 \cite{ghosal2017fundamentals}, since  when  $\phi=o(n^{-\frac{v-\eta}{(2\eta+d)v}}),\,\varepsilon_n=
   \left( {\phi^{-d}}/{n}\right)^{\frac{v}{2v+d}}$ satisfies the rate equation $\varphi_{w_0}(\epsilon_n) \le n \epsilon_n^2$ (by the statement before Theorem \ref{rescale matern theorem}). Thus, $\varphi_{w_0}(\delta_n) \ge \varphi_0(\delta_n) \gtrsim \phi^{-d} {\delta_n}^{-d/v} \ge C_0 n\epsilon_n^2$ for sufficiently large constant $C_0$. If $\delta_n=(\frac{\phi^{-d}}{n})^{\frac{v}{2v+d}} \succ n^{-\frac{\eta}{2\eta+d}}$, then the contraction rate is suboptimal since $\delta_n$, the lower bound of contraction rate   has larger order than the optimal rate and the last assertion follows when $\phi^{-d} \gtrsim n$, $\delta_n \ge 1$.
\end{proof}

\subsection{Proof of Lemma \ref{CH small ball}}
\begin{proof}
 Let $g= \psi m$, and for  the minimal choice of $\psi$ as in (\ref{RKHS of rescaled CH}), we have for the CH covariance:
\begin{displaymath}
\|\mathcal{F}_{(\alpha,\beta)} \psi\|_{\mathbb{H}^{\alpha,\beta}}^2  =\int |g(\lambda)|^2 (m_{CH}^{\alpha,\beta})^{-1}(\lambda) d\lambda\\
=\int |g(\lambda)|^2 (1+\lambda^2)^{(v+d/2)}\frac{(1+\lambda^2)^{-(v+d/2)}}{m_{CH}^{\alpha,\beta}(\lambda)} d\lambda.
\end{displaymath}
By Lemma \ref{inequality about ch spectral density}, we have $\frac{(1+\lambda^2)^{-(v+d/2)}}{m_{CH}^{\alpha,\beta}(\lambda)} \gtrsim \frac{\Gamma(\alpha) \beta^{2v}}{\Gamma(\alpha+v)}$, then $\|\mathcal{F}_{(\alpha,\beta)} \psi\|_{\mathbb{H}^{\alpha,\beta}}^2  \geq$\\ $C \frac{\Gamma(\alpha) \beta^{2v}}{\Gamma(\alpha+v)} \|g\|^2_{2,2,v+d/2}$, and the unit ball of RKHS is contained in the Sobolev ball  of radius $\sqrt{\frac{\Gamma(\alpha+v)}{\Gamma(\alpha) \beta^{2v}}}$ (up to a constant) of order $v+d/2$.  By  Theorem 2.7.4 in \cite{wellner2013weak}, the metric entropy of  such a Sobolev ball is bounded by a constant times $(\frac{\Gamma(\alpha+v)}{\Gamma(\alpha) \beta^{2v}})^{\frac{d/2}{v+d/2}}  \varepsilon^{-\frac{d}{v+d/2}}$. By Theorem 1.2 of \cite{li1999approximation},
\begin{equation}\label{small ball of ch invariation}
\begin{split}
    \varphi_0(\varepsilon) &\lesssim \varepsilon ^{-\frac{2 \frac{d}{v+d/2}}{2-\frac{d}{v+d/2}}}\left[\left(\frac{\Gamma(\alpha+v)}{\Gamma(\alpha) \beta^{2v}} \right)^{\frac{d/2}{v+d/2}} \right]^{\frac{2v+d}{2v}}=\varepsilon^{-d/v} \left(\frac{\Gamma(\alpha+v)}{\Gamma(\alpha) \beta^{2v}}\right)^{d/(2v)}.
    \end{split}
\end{equation}
From the proof of  Proposition 3.1 of \citet{li1999approximation}, this bound holds for all $\varepsilon >0$ satisfying,
\begin{displaymath}
    \left[\frac{\Gamma(\alpha) \beta^{2v}}{\Gamma(\alpha+v)} \right]^{\frac{d/2}{v+d/2}} \lesssim (\varphi_0(\varepsilon/2))^{\frac{d}{2(v+d/2)}} \varepsilon^{-\frac{d}{v+d/2}}.
\end{displaymath}
Similar to the proof in Theorem \ref{lower bound of rescaled matern}, by Lemma \ref{CH spectral density tail behavior},  for rescaled CH class,  when $\|\lambda\| \ge (\alpha+v-1)\beta^{-1}$, its spectral density satisfies:
$$
m^{\alpha,\beta}_{CH}(\lambda) \gtrsim \frac{\Gamma(\alpha +v)}{\Gamma(\alpha )\beta^{2v}} \|\lambda\|^{2v+d},
$$
and  we have:
$$
P\left(\sup_{t \in \mathcal{T}} |W_t^{\alpha,\beta}| \le \varepsilon/2\right) \le \exp\left(-C \left(\frac{\Gamma(\alpha+v)}{\Gamma(\alpha) \beta^{2v}}\right)^{d/(2v)}\varepsilon^{-d/v}\right),
$$
where $\varepsilon$ and $C$ only depend on $v$ and $d$.  When $\frac{\Gamma(\alpha+v)}{\Gamma(\alpha) \beta^{2v}} >1$, we have,
$$
\varphi_0(\varepsilon/2) = - \log P\left(\sup_{t \in \mathcal{T}} |W_t^{\alpha,\beta}| \le \varepsilon/2\right)  \ge C \left(\frac{\Gamma(\alpha+v)}{\Gamma(\alpha) \beta^{2v}}\right)^{\frac{d}{2v}}\varepsilon^{-d/v} \ge C \varepsilon^{-d/v}.
$$
Since the right hand side is independent of $\alpha, \beta$, it follows that (\ref{small ball of ch invariation}) holds for all $\varepsilon$ in an interval independent
of $\alpha, \beta$. 
\end{proof}

\subsection{Proof of Lemma \ref{CH decentering}}
\begin{proof}
Here,  we use the exact construction in the proof of Lemma \ref{matern decentering}, but let $h=\zeta_{\beta^{\theta}}* w_0$, $\theta>1$.
By Lemma \ref{CH spectral density tail behavior}, when
$\theta > \frac{v+d/2}{v+d/2-\eta}$, $\beta^{-1} \gtrsim \ln n$ and $\alpha \lesssim \sqrt{\ln \ln n}$,  we have,
\begin{equation}
\begin{split}
\|h\|_\mathbb{H^{\alpha,\beta}}^2 &= \left\| 2 \pi \mathcal{F}_{(\alpha,\beta)}\left( \frac{\hat{\zeta}(\beta^{\theta} \lambda)\hat{w_0}(\lambda)}{m_{CH}^{\alpha,\beta}(\lambda)}\right)\right\|_\mathbb{H^{\alpha,\beta}}^2 \\
& \leq  \tilde{D}_{w_0} \cdot \int|\hat{\varphi}(\beta^{\theta} \lambda)\hat{w_0}(\lambda)|^2 \frac{1}{{m_{CH}^{\alpha,\beta}}(\lambda)} d\lambda \\ 
&\leq  \tilde{D}_{w_0} \cdot  \sup_\lambda[(1+{\|\lambda\|^2})^{-\eta} ({m_{CH}^{\alpha,\beta}})^{-1}(\lambda) |\hat{\varphi}(\beta^{\theta} \lambda)|^2] \times \|w_0\|^2_{2,2,\eta}\\
& =  \tilde{D}_{w_0} \cdot \sup_{\|\lambda\| \leq M/\beta^{\theta}}[(1+{\|\lambda\|^2})^{-\eta}({m_{CH}^{\alpha,\beta}})^{-1}(\lambda) |\hat{\varphi}(\beta^{\theta}\lambda)|^2] \times \|w_0\|^2_{2,2,\eta}\\
&\leq  {D}_{w_0} \cdot \sup_{\|\lambda\| \leq M/\beta^{\theta}}[(1+{\|\lambda\|^2})^{-\eta}({m_{CH}^{\alpha,\beta}})^{-1}(\lambda) ] \times \|w_0\|^2_{2,2,\eta}\\
&\lesssim  \max \left\{\frac{\Gamma(\alpha)}{\Gamma(\alpha-d/2)\beta^d}, \frac{\Gamma(\alpha) e^\alpha}{\beta^{-2v}\Gamma(\alpha-d/2)\alpha}\left(\frac{\alpha+v-1}{\beta^2}\right)^{v+d/2-\eta},\right.\\
&\left.\qquad(\beta^{2\theta})^{-v-d/2+\eta} \frac{\Gamma(\alpha) \beta^{2v}}{\Gamma(\alpha+v)}\right\}                                         \\
&\leq  (\beta^{2\theta})^{-v-d/2+\eta} \frac{\Gamma(\alpha) \beta^{2v}}{\Gamma(\alpha+v)},
\end{split}
\end{equation}
where $\tilde{D}_{w_0}$ and ${D}_{w_0}$ depend only on $w_0$. 
\end{proof}

\subsection{Proof of Theorem \ref{lower bound of rescaled CH}}
\begin{proof}
For the rescaled CH class,  by Lemma \ref{CH spectral density tail behavior}, when $\|\lambda\| \ge (\alpha+v-1)\beta^{-1}$, its spectral density satisfies:
$$
m^{\alpha,\beta}_{CH}(\lambda) \ge c_0 \beta^{-2v} \|\lambda\|^{-(2v+d)},
$$
where $c_0$ does not depend on $\alpha, \beta$. Then following the same steps in the proof of Theorem \ref{lower bound of rescaled matern}, completes the present proof.
\end{proof}

\subsection{Proof of Theorem \ref{adaptive matern}}
\begin{proof}
We consider a prior on $A=1/\phi$ with Lebesgue density $\tilde{g}_A(\cdot)$   satisfying the condition: 
\begin{equation}
\tilde{C}_1 a^p \exp \left(-\tilde{D}_1 a^{kd} \right) \leq \tilde{g}_A(a) \leq \tilde{C}_2 a^p \exp \left(-\tilde{D}_2 a^{kd}  \right) ,
\end{equation}
for positive constants $\tilde{C}_1, \tilde{D}_1, \tilde{C}
_2, \tilde{D}_2$, non-negative constants $p, k$ and all  sufficiently large $a>0$,  and when $k=1$ this prior  is the same as the prior satisfying  (\ref{gamma prior tail }).
Let $f$ be the pdf of the prior on $\phi$.

Consider the condition in (\ref{adaptive condition1}).  By Proposition 11.19 of \cite{ghosal2017fundamentals}, we have, 
\begin{equation}
    P\left(\left\|W^{\phi}-w_0\right\|_{\infty} \leq 2 \varepsilon\right) \geq e^{-\varphi_{w_0}^{\phi}(\varepsilon)},
\end{equation}
where $\varphi_{w_0}^{\phi}(\varepsilon)$ is the small ball exponent $\varphi_{w_0}(\varepsilon)$ in Lemma \ref{matern small ball}.

By Lemma \ref{matern small ball} we have that $\varphi_0^{\phi}(\varepsilon) \leq C  \varepsilon^{-d/v} \phi^{-d}$ for $\phi<\phi_0<1/2$ and $\varepsilon<\varepsilon_0$, where the constants $\phi_0, \varepsilon_0, C$ depend only on  $w_0$ and $\mu$. By Lemmas \ref{matern small ball} and \ref{matern decentering} (taking $\theta=v/(v-\eta)$ in Lemma \ref{matern decentering}),  for $\phi  <\phi_0, \varepsilon<\varepsilon_0$ and $\varepsilon \asymp \phi^{\frac{\eta v}{v-\eta}}$ (so that $ \phi^{\theta \eta} \lesssim \varepsilon$), we have:
$$
\varphi_{w_0}^\phi(\varepsilon) \leq C_1  \varepsilon^{-d/v} \phi^{-d} +
D \phi^{-\frac{vd}{v-\eta}}
\leq K \varepsilon^{-d/v} \phi^{-d},
$$
for $K$ depending on $\phi_0, \mu$ and $d$ only. Therefore, for $\varepsilon<\varepsilon_0 \wedge C_1 \phi_0^{\frac{v\eta}{v-\eta}}$ (so that $ (\varepsilon/ C_1)^{\frac{v-\eta}{v\eta}} \le \phi_0$), we have:
$$
\begin{aligned}
P\left(\left\|W^A-w_0\right\|_{\infty} \leq 2 \varepsilon\right)&= \int_0^{\infty} P\left(\left\|W^{\phi}-w_0\right\|_{\infty} \leq 2 \varepsilon\right) f({\phi}) d {\phi} \\
& \geq \int_0^{\infty} e^{-\varphi_{w_0}^{\phi}(\varepsilon)} f({\phi}) d {\phi} \\
& \geq \int_{(\varepsilon/(2 C_1))^{\frac{v-\eta}{v\eta}}}^{(  \varepsilon/C_1)^{\frac{v-\eta}{v\eta}}} e^{-K \varepsilon^{-d/v} \phi^{-d}}f({\phi}) d {\phi}  \\
& \geq C_2 e^{-K_2 \varepsilon^{-d/\eta}}  \int_{(\varepsilon/(2 C_1))^{\frac{v-\eta}{v\eta}}}^{(  \varepsilon/C_1)^{\frac{v-\eta}{v\eta}}}f({\phi}) d {\phi}\\
&= C_2 e^{-K_2 \varepsilon^{-d/\eta}}  \int_{(C_1/\varepsilon )^{\frac{v-\eta}{v\eta}}}^{(  2C_1/\varepsilon)^{\frac{v-\eta}{v\eta}}}\tilde{g}_A({a}) d {a}\\
& \ge C_2 e^{-K_2 \varepsilon^{-d/\eta}} (C_1/\varepsilon)^{\frac{(p+1)(v-\eta)}{v\eta}} \exp{(-D_1(C_1/\varepsilon)^{\frac{kd(v-\eta)}{v\eta}})}\\
& \geq C_3 e^{-K_3 \varepsilon^{-d/\eta}},
\end{aligned}
$$
for constant $K_3$ that depends only on $C_1, D, D_1, d, \eta,K$ and the last inequality in the previous display holds because $k \leq \frac{v}{v-\eta}$. Then we have that \\$P\left(\left\|W^A-w_0\right\|_{\infty} \leq \varepsilon_n\right) \geq \exp \left(-n \varepsilon_n^2\right)$ for $\varepsilon_n=C_4 n^{-\eta /(2\eta+d )}$ and sufficiently large $n$.

Next, consider the condition in (\ref{adaptive condition2}). Let
$\mathbb{B}_1$ be the unit ball of $C(\mathcal{T})$ and set
\begin{equation}
B=B_{M, r, \delta, \varepsilon}=\left(M ({\frac{\delta}{r}})^{d/2} \mathbb{H}_1^r+\varepsilon \mathbb{B}_1\right) \cup\left(\bigcup_{{\phi} > \delta}\left(M \mathbb{H}_1^{\phi}\right)+\varepsilon \mathbb{B}_1\right),
\end{equation}
 where positive constants $M, r, \delta, \varepsilon$ are to be determined.

By Lemma \ref{rescaled hilbert inclusion} the set $B$ contains the set $M \mathbb{H}_1^{\phi}+\varepsilon \mathbb{B}_1$ for any ${\phi} \in[r, \delta]$. By the definition of $B$,  for ${\phi}>\delta$  this is true. By Borell's inequality (Proposition 11.17 in \cite{ghosal2017fundamentals}) and the fact that $e^{-\varphi_0^{\phi}(\varepsilon)}=P\left(\sup _{t \in \mathcal{T}/\phi}\left|W_t\right| \leq \varepsilon\right)$ is increasing in ${\phi}$,  one has for any ${\phi} \geq r$,
\begin{equation}\label{borell ineq}
\begin{aligned}
P\left(W^{\phi} \notin B\right) & \leq P\left(W^{\phi} \notin M \mathbb{H}_1^{\phi}+\varepsilon \mathbb{B}_1\right) \leq 1-\Phi\left(\Phi^{-1}\left(e^{-\varphi_0^{\phi}(\varepsilon)}\right)+M\right) \\
& \leq 1-\Phi\left(\Phi^{-1}\left(e^{-\varphi_0^r(\varepsilon)}\right)+M\right).
\end{aligned}
\end{equation}
By Lemma 4.10 of \cite{van2009adaptive}, when 
\begin{equation} \label{M's condition}
M \geq 4 \sqrt{\varphi_0^r(\varepsilon)}, \text { and } e^{-\varphi_0^r(\varepsilon)}<1 / 4 ,
\end{equation}
we note that $e^{-\varphi_0^r(\varepsilon)} \leq e^{-\varphi_0^1(\varepsilon)}$ for $r<1$ and is  smaller than $1 / 4$ if $\varepsilon$ is smaller than some fixed $\varepsilon_1$, so  
$$
M \geq-2 \Phi^{-1}\left(e^{-\varphi_0^r(\varepsilon)}\right).
$$
Then the right-hand side of (\ref{borell ineq}) is bounded by $1-\Phi(M / 2) \leq e^{-M^2 / 8}$. Therefore, by Lemma \ref{matern small ball} the inequalities (\ref{M's condition})  are satisfied if,
\begin{equation}\label{M range}
M^2 \geq 16 C_5 \varepsilon^{-d/v} r^{-d}, \quad r<1, \quad \varepsilon<\varepsilon_1 \wedge \varepsilon_0 .
\end{equation}
Then by Lemma 4.9 in \cite{van2009adaptive}, the following inequality holds if  $M, r, \delta, \varepsilon$ satisfy  (\ref{M range}):
\begin{equation}\label{M,r range}
\begin{aligned}
P\left(W^A \notin B\right) & \leq P(\phi<r)+\int_r^\infty P\left(W^{\phi} \notin B\right) f({\phi}) d {\phi} \\
& \leq \frac{2 C_2 r^{-p+kd-1} e^{-D_2 r^{-kd} }}{D_2 d }+e^{-M^2 / 8}.
\end{aligned}
\end{equation}
By   (\ref{M,r range}), to show the condition (\ref{adaptive condition2}) it suffices to verify the following inequalities:
\begin{equation} \label{M range 2}
\begin{aligned}
D_2 (1/r)^{kd} & \geq 8 n \varepsilon_n^2, \\
(1/r)^{p-kd+1} & \leq e^{4 n \varepsilon_n^2,} \\
M^2 & \geq 32 n \varepsilon_n^2.
\end{aligned}
\end{equation}
The choice:
\begin{equation} \label{solution of M,r}
    \begin{aligned}
r=r_n&=({D_2}/{8})^{1/(kd)}n^{-1/(k(2 \eta+d))},\\
M=M_n&=(32n^{{2d}/(2 \eta+d)})^{1/2},\\
\varepsilon=\varepsilon_n&=n^{-\frac{\eta}{2\eta+d}},\\
\end{aligned}
\end{equation}
 satisfies these inequalities while also satisfying (\ref{M range}) when $k \ge \frac{v}{2v-\eta}$.

Finally, consider the condition in  (\ref{adaptive condition3}). By the proof of Lemma \ref{matern small ball}, for $M (\frac{\delta}{r})^{d/2}>2 \varepsilon$ and $r< \phi_0$, we have:
$$
\begin{aligned}
\log N\left(2 \varepsilon, M (\frac{\delta}{r})^{d/2} \mathbb{H}_1^r+\varepsilon \mathbb{B}_1,\|\cdot\|_{\infty}\right) & \leq \log N\left(\varepsilon, M (\frac{\delta}{r})^{d/2} \mathbb{H}_1^r,\|\cdot\|_{\infty}\right) \\
& \leq C \left( \frac{M}{\varepsilon} (\frac{r}{\delta})^{d/2}r^{-v}  \right)^{\frac{d}{v+d/2}} .
\end{aligned}
$$

By Lemma \ref{uniform distance within hilbert}, every element of $M \mathbb{H}_1^{\phi}$ for $\phi>\delta$ is within uniform distance $ \sqrt{d} \tau M/\delta$ (with $\tau=(\int \|\lambda\|^2 d \mu)^{1/2}$) of a constant function and this constant is contained in the interval $[-M \sqrt{\|\mu\|}, M \sqrt{\|\mu\|}]$. Then for $\varepsilon> \sqrt{d} \tau M/\delta$,
$$
N\left(2 \varepsilon, \bigcup_{\phi>\delta}\left(M \mathbb{H}_1^{\phi}\right)+\varepsilon \mathbb{B}_1,\|\cdot\|_{\infty}\right) \leq N(\varepsilon,[-M \sqrt{\|\mu\|}, M \sqrt{\|\mu\|}],|\cdot|) \leq \frac{2 M \sqrt{\|\mu\|}}{\varepsilon} .
$$

Now, with the choice $\delta=(2 \sqrt{d} \tau M/\varepsilon)^2$, combining the last two displays, and using the inequality $\log (x+y) \leq \log x+2 \log y$ for $x \geq 1, y \geq 2$, we obtain,
\begin{equation}\label{adaptive entropy inequality}
\begin{aligned}
&\log N\left(2 \varepsilon, B,\|\cdot\|_{\infty}\right) \\
&\leq  \log \left[N\left(2 \varepsilon, M (\frac{\delta}{r})^{d/2} \mathbb{H}_1^r+\varepsilon \mathbb{B}_1,\|\cdot\|_{\infty}\right) +  N\left(2 \varepsilon, \bigcup_{\phi>\delta}\left(M \mathbb{H}_1^{\phi}\right)+\varepsilon \mathbb{B}_1,\|\cdot\|_{\infty}\right)\right] \\
& \leq C \left( \frac{M}{\varepsilon} (\frac{r}{\delta})^{d/2}r^{-v}  \right)^{\frac{d}{v+d/2}} 
+ 2\log(\frac{2 M \sqrt{\|\mu\|}}{\varepsilon}) .
\end{aligned}
\end{equation}
This inequality is valid for any $B=B_{M, r, \delta, \varepsilon}$ with $\delta=(2 \sqrt{d} \tau M/\varepsilon)^2$, and any $M, r, \varepsilon$ with:
\begin{equation}
 \quad r < \phi_0 (<1/2), \quad M (\frac{\delta}{r})^{d/2}>2 \varepsilon.
\end{equation}
We find that the solution  in (\ref{solution of M,r}) satisfies these inequalities, and with this solution, if we also have $k\ge \frac{v-d/2}{v-\eta+d\eta+d(d-1)}$ and $v \ge \eta$, the right hand side of (\ref{adaptive entropy inequality}) is bounded by $n \varepsilon_n^2$, which verifies the condition in (\ref{adaptive condition3}).

In sum, if: 
\begin{equation}\label{k condition}
\max\left\{\frac{v}{2v-\eta},\frac{v-d/2}{v-\eta+d\eta+d(d-1)}\right\}\leq k \leq \frac{v}{v-\eta},
\end{equation}
then conditions (\ref{adaptive condition1})--(\ref{adaptive condition3}) are satisfied. Condition (\ref{k condition}) on $k$  can be simplified to $1 \leq k \leq \frac{v}{v-\eta}$ when $v \ge \eta$, and we can take $k=1$ to complete the proof.
\end{proof}

\subsection{Proof of Theorem \ref{adaptive CH}}
\begin{proof}
We consider a prior on $A=1/\beta$ with Lebesgue density $\tilde{g}_A(\cdot)$  satisfying the condition: 
\begin{equation}
\tilde{C}_1 a^p \exp \left(-\tilde{D}_1 a^{kd} \right) \leq \tilde{g}_A(a) \leq \tilde{C}_2 a^p \exp \left(-\tilde{D}_2 a^{kd}  \right) ,
\end{equation}
for positive constants $\tilde{C}_1, \tilde{D}_1, \tilde{C}
_2, \tilde{D}_2$, non-negative constants $p, k$ and all  sufficiently large $a>0$,  and when $k=1$ this prior  is the same as the prior satisfying  (\ref{gamma prior tail }).
Let $f$ be the pdf of the prior on $\beta$.

Consider the condition in (\ref{adaptive condition1}).  By Proposition 11.19 of \cite{ghosal2017fundamentals}, we have, 
\begin{equation}
    P\left(\left\|W^{\alpha,\beta}-w_0\right\|_{\infty} \leq 2 \varepsilon\right) \geq e^{-\varphi_{w_0}^{\alpha,\beta}(\varepsilon)},
\end{equation}
where $\varphi_{w_0}^{\alpha,\beta}(\varepsilon)$ is the small ball exponent $\varphi_{w_0}(\varepsilon)$ in Lemma \ref{CH small ball}.

By Lemma \ref{CH small ball}, when $\alpha>d/2+1$, we have that $\varphi_0^{\alpha,\beta}(\varepsilon) \leq C  \varepsilon^{-d/v} \beta^{-d}$ for $\beta<\beta_0<1/2$ and $\varepsilon<\varepsilon_0$, where the constants $\beta_0, \varepsilon_0, C$ depend only on  $w_0$ and $\mu$. By Lemmas \ref{CH small ball} and \ref{CH decentering} (with $\theta=v/(v-\eta)$ in Lemma \ref{CH decentering}),  for $\beta  <\beta_0, \varepsilon<\varepsilon_0$ and $\varepsilon \asymp \beta^{\frac{\eta v}{v-\eta}}$ (so that $ \beta^{\theta \eta} \lesssim \varepsilon$), we have:
$$
\varphi_{w_0}^{\alpha,\beta}(\varepsilon) \leq C_1  \varepsilon^{-d/v} \beta^{-d} +
D \beta^{-\frac{vd}{v-\eta}}
\leq K \varepsilon^{-d/v} \beta^{-d},
$$
for $K$ depending on $\beta_0, \mu$ and $d$ only. Therefore, for $\varepsilon<\varepsilon_0 \wedge C_1 \beta_0^{\frac{v\eta}{v-\eta}}$(so that $ (\varepsilon/ C_1)^{\frac{v-\eta}{v\eta}} \le \beta_0$), 
$$
\begin{aligned}
P\left(\left\|W^A-w_0\right\|_{\infty} \leq 2 \varepsilon\right)&= \int_0^{\infty} P\left(\left\|W^{\alpha,\beta}-w_0\right\|_{\infty} \leq 2 \varepsilon\right) f({\beta}) d {\beta} \\
& \geq \int_0^{\infty} e^{-\varphi_{w_0}^{\alpha,\beta}(\varepsilon)} f({\beta}) d {\beta} \\
& \geq \int_{(\varepsilon/(2 C_1))^{\frac{v-\eta}{v\eta}}}^{(  \varepsilon/C_1)^{\frac{v-\eta}{v\eta}}} e^{-K \varepsilon^{-d/v} \beta^{-d}}f({\beta}) d {\beta}  \\
& \geq C_2 e^{-K_2 \varepsilon^{-d/\eta}}  \int_{(\varepsilon/(2 C_1))^{\frac{v-\eta}{v\eta}}}^{(  \varepsilon/C_1)^{\frac{v-\eta}{v\eta}}}f({\beta}) d {\beta}\\
&= C_2 e^{-K_2 \varepsilon^{-d/\eta}}  \int_{(C_1/\varepsilon )^{\frac{v-\eta}{v\eta}}}^{(  2C_1/\varepsilon)^{\frac{v-\eta}{v\eta}}}\tilde{g}_A({a}) d {a}\\
& \ge C_2 e^{-K_2 \varepsilon^{-d/\eta}} (C_1/\varepsilon)^{\frac{(p+1)(v-\eta)}{v\eta}} \exp{(-D_1(C_1/\varepsilon)^{\frac{kd(v-\eta)}{v\eta}})}\\
& \geq C_3 e^{-K_3 \varepsilon^{-d/\eta}},
\end{aligned}
$$
for constant $K_3$ that depends only on $C_1, D, D_1, d, \eta,K$ and the last inequality in the previous display holds for $k \leq \frac{v}{v-\eta}$. Then we have that \\$P\left(\left\|W^A-w_0\right\|_{\infty} \leq \varepsilon_n\right) \geq \exp \left(-n \varepsilon_n^2\right)$ for $\varepsilon_n=C_4 n^{-\eta /(2\eta+d )}$ and sufficiently large $n$.

Next, consider the condition in (\ref{adaptive condition2}). Let
$\mathbb{B}_1$ be the unit ball of $C(\mathcal{T})$ and set
\begin{equation}
B=B_{M, r, \delta, \varepsilon}=\left(M ({\frac{\delta}{r}})^{d/2} \mathbb{H}_1^{\alpha,r}+\varepsilon \mathbb{B}_1\right) \cup\left(\bigcup_{{\beta} > \delta}\left(M \mathbb{H}_1^{\alpha,\beta}\right)+\varepsilon \mathbb{B}_1\right),
\end{equation}
 where positive constants $M, r, \delta, \varepsilon$ are to be determined.

By Lemma \ref{rescaled hilbert inclusion} the set $B$ contains the set $M \mathbb{H}_1^{\alpha,\beta}+\varepsilon \mathbb{B}_1$ for any ${\beta} \in[r, \delta]$. By the definition of $B$,  for ${\beta}>\delta$  this is true. By Borell's inequality (Proposition 11.17 in \cite{ghosal2017fundamentals}) and the fact that $e^{-\varphi_0^{\alpha,\beta}(\varepsilon)}=P\left(\sup _{t \in \mathcal{T}/\beta}\left|W^{\alpha,1}_t\right| \leq \varepsilon\right)$ is increasing in ${\beta}$,  one has for any ${\beta} \geq r$,
\begin{equation}\label{CH borell ineq}
\begin{aligned}
P\left(W^{\alpha,\beta} \notin B\right) & \leq P\left(W^{\alpha,\beta} \notin M \mathbb{H}_1^{\alpha,\beta}+\varepsilon \mathbb{B}_1\right) \leq 1-\Phi\left(\Phi^{-1}\left(e^{-\varphi_0^{\alpha,\beta}(\varepsilon)}\right)+M\right) \\
& \leq 1-\Phi\left(\Phi^{-1}\left(e^{-\varphi_0^{\alpha,r}(\varepsilon)}\right)+M\right).
\end{aligned}
\end{equation}
By Lemma 4.10 of \cite{van2009adaptive}, when 
\begin{equation} \label{CH M's condition}
M \geq 4 \sqrt{\varphi_0^{\alpha,r}(\varepsilon)}, \text { and } e^{-\varphi_0^{\alpha,r}(\varepsilon)}<1 / 4 ,
\end{equation}
we note that $e^{-\varphi_0^{\alpha,r}(\varepsilon)} \leq e^{-\varphi_0^{\alpha,1}(\varepsilon)}$ for $r<1$ and is  smaller than $1 / 4$ if $\varepsilon$ is smaller than some fixed $\varepsilon_1$, so  
$$
M \geq-2 \Phi^{-1}\left(e^{-\varphi_0^{\alpha,r}(\varepsilon)}\right).
$$
Then the right-hand side of (\ref{CH borell ineq}) is bounded by $1-\Phi(M / 2) \leq e^{-M^2 / 8}$. Therefore, by Lemma \ref{CH small ball} the inequalities (\ref{CH M's condition})  are satisfied if,
\begin{equation}\label{CH M range}
M^2 \geq 16 C_5 \varepsilon^{-d/v} r^{-d}, \quad r<1, \quad \varepsilon<\varepsilon_1 \wedge \varepsilon_0 .
\end{equation}
Then by Lemma 4.9 in \cite{van2009adaptive}, the following inequality holds if  $M, r, \delta, \varepsilon$ satisfy  (\ref{CH M range})
\begin{equation}\label{CH M,r range}
\begin{aligned}
P\left(W^A \notin B\right) & \leq P(\beta<r)+\int_r^\infty P\left(W^{\alpha,\beta} \notin B\right) f({\beta}) d {\beta} \\
& \leq \frac{2 C_2 r^{-p+kd-1} e^{-D_2 r^{-kd} }}{D_2 d }+e^{-M^2 / 8}.
\end{aligned}
\end{equation}
By   (\ref{CH M,r range}), to show the condition (\ref{adaptive condition2}) it suffices to verify the following inequalities:
\begin{equation} \label{CH M range 2}
\begin{aligned}
D_2 (1/r)^{kd} & \geq 8 n \varepsilon_n^2, \\
(1/r)^{p-kd+1} & \leq e^{4 n \varepsilon_n^2,} \\
M^2 & \geq 32 n \varepsilon_n^2.
\end{aligned}
\end{equation}
The choice:
\begin{equation} \label{CH solution of M,r}
    \begin{aligned}
r=r_n&=({D_2}/{8})^{1/(kd)}n^{-1/(k(2 \eta+d))}\\
M=M_n&=(32n^{{2d}/(2 \eta+d)})^{1/2}\\
\varepsilon=\varepsilon_n&=n^{-\frac{\eta}{2\eta+d}}\\
\end{aligned}
\end{equation}
 satisfies these inequalities while also satisfying (\ref{CH M range 2}) when $k \ge \frac{v}{2v-\eta}$.

Finally, consider the condition in  (\ref{adaptive condition3}). By the proof of Lemma \ref{CH small ball}, for $M (\frac{\delta}{r})^{d/2}>2 \varepsilon$ and $r< \beta_0$,
$$
\begin{aligned}
\log N\left(2 \varepsilon, M (\frac{\delta}{r})^{d/2} \mathbb{H}_1^{\alpha,r}+\varepsilon \mathbb{B}_1,\|\cdot\|_{\infty}\right) & \leq \log N\left(\varepsilon, M (\frac{\delta}{r})^{d/2} \mathbb{H}_1^{\alpha,r},\|\cdot\|_{\infty}\right) \\
& \leq C \left( \frac{M}{\varepsilon} (\frac{\delta}{r})^{d/2}r^{-v}  \right)^{\frac{d}{v+d/2}} .
\end{aligned}
$$

By Lemma \ref{uniform distance within hilbert}, every element of $M \mathbb{H}_1^{\alpha,r}$ for $\beta>\delta$ is within uniform distance $ \sqrt{d} \tau M/\delta$ (let $\tau=(\int \|\lambda\|^2 d \mu)^{1/2}$) of a constant function and this constant is contained in the interval $[-M \sqrt{\|\mu\|}, M \sqrt{\|\mu\|}]$. Then for $\varepsilon> \sqrt{d} \tau M/\delta$,
$$
N\left(2 \varepsilon, \bigcup_{\beta>\delta}\left(M \mathbb{H}_1^{\alpha,\beta}\right)+\varepsilon \mathbb{B}_1,\|\cdot\|_{\infty}\right) \leq N(\varepsilon,[-M \sqrt{\|\mu\|}, M \sqrt{\|\mu\|}],|\cdot|) \leq \frac{2 M \sqrt{\|\mu\|}}{\varepsilon} .
$$

Now, with the choice $\delta=(2 \sqrt{d} \tau M/\varepsilon)^2$, combining the last two displays, and using the inequality $\log (x+y) \leq \log x+2 \log y$ for $x \geq 1, y \geq 2$, we obtain,
\begin{equation}\label{CH adaptive entropy inequality}
\begin{aligned}
&\log N\left(2 \varepsilon, B,\|\cdot\|_{\infty}\right) \\
&\leq  \log \left[N\left(2 \varepsilon, M (\frac{\delta}{r})^{d/2} \mathbb{H}_1^{\alpha,r}+\varepsilon \mathbb{B}_1,\|\cdot\|_{\infty}\right) +  N\left(2 \varepsilon, \bigcup_{\beta>\delta}\left(M \mathbb{H}_1^{\alpha,\beta}\right)+\varepsilon \mathbb{B}_1,\|\cdot\|_{\infty}\right)\right] \\
& \leq C \left( \frac{M}{\varepsilon} (\frac{r}{\delta})^{d/2}r^{-v}  \right)^{\frac{d}{v+d/2}} 
+ 2\log(\frac{2 M \sqrt{\|\mu\|}}{\varepsilon}) .
\end{aligned}
\end{equation}
This inequality is valid for any $B=B_{M, r, \delta, \varepsilon}$ with $\delta=(2 \sqrt{d} \tau M/\varepsilon)^2$, and any $M, r, \varepsilon$ with:
\begin{equation}
 \quad r < \beta_0 (<1/2), \quad M (\frac{\delta}{r})^{d/2}>2 \varepsilon.
\end{equation}
We find that the solution  in (\ref{CH solution of M,r}) satisfies these inequalities, and with this solution, if we also have $k\ge \frac{v-d/2}{v-\eta+d\eta+d(d-1)}$ and $v \ge \eta$, the right hand side of (\ref{CH adaptive entropy inequality}) is bounded by $n \varepsilon_n^2$, which verifies the condition in (\ref{adaptive condition3}).

In sum, if 
\begin{equation}\label{CH k condition}
\max\left\{\frac{v}{2v-\eta},\frac{v-d/2}{v-\eta+d\eta+d(d-1)}\right\}\leq k \leq \frac{v}{v-\eta},
\end{equation}
then conditions (\ref{adaptive condition1})--(\ref{adaptive condition3}) are satisfied. Condition (\ref{CH k condition}) on $k$  can be simplified to $1 \leq k \leq \frac{v}{v-\eta}$ when $v \ge \eta$, and taking $k=1$, we complete the proof.
\end{proof}

\subsection{Proof of Theorem \ref{rescale anisotropic theorem}}
\begin{proof}
There exists constant $C>0$  such that 
 $$ 1/C \cdot m_M^{{\lambda_{\max}}}(\bm{\lambda})  \leq m_M^{\bm{B}}(\bm{\lambda}) \leq C m_M^{{\lambda_{\max}}}(\bm{\lambda}), $$
 and,
 $$
 1/C \cdot m^{\alpha,\lambda_{\max}}_{CH}(\bm{\lambda})  \leq m^{\alpha,\bm{B}}_{CH}(\bm{\lambda}) \leq C m^{\alpha,\lambda_{\max}}_{CH}(\bm{\lambda}).
 $$
Then the proof of this theorem follows similarly to the proofs of  Theorem \ref{rescale matern theorem} and Theorem \ref{rescale CH theorem}.
\end{proof}

\section{Ancillary Results} \label{Ancillary}
First,  we recapture some useful results for the CH covariance, as introduced in \cite{ma2022beyond}.

\begin{enumerate}

\item The CH covariance function can be  obtained as a mixture of the Mat\'ern class over its lengthscale parameter $\phi$ as: 
$$C(h;v,\alpha,\beta,\sigma^2) := \int_0^{\infty} M(h;v,\phi,\sigma^2) \pi(\phi^2;\alpha,\beta) d \phi^2,$$
where $\phi^2 \sim IG(\alpha,\beta)$,
is given an inverse gamma mixing density. \cite{ma2022beyond} prove that this is a valid covariance function on $\mathbb{R}^d$ for all positive integers $d$, where the Mat\'ern and CH covariance functions are as defined in Equations~\eqref{eq:Matern}--\eqref{eq:CH}.

\item The spectral density $m^{\alpha,\beta}_{CH}(\lambda)$ of the CH covariance function is given by \citet{ma2022beyond} as:
\begin{displaymath}
\begin{split}
m^{\alpha,\beta}_{CH}(\lambda) &=  \frac{\sigma^2 2^{v-\alpha} v^v \beta^{2\alpha}}{\pi^{d/2}\Gamma(\alpha)}\int_0^{\infty} (2v\phi^{-2}+\lambda^2)^{-v-\frac{d}{2}}  \phi^{-2(v+\alpha+1)} \exp{(-\frac{\beta^2}{2\phi^2})d\phi^2}. \\
\end{split}
\end{displaymath}
We also note the spectral density $m^{\phi}_M(\lambda)$ of the Mat\'ern covariance function is \citep{stein1999interpolation}:
$$
m^{\phi}_M(\lambda) =\frac{\sigma^2(\sqrt{2 v} / \phi)^{2 v}}{\pi^{d / 2}\left((\sqrt{2 v} / \phi)^2+\lambda^2\right)^{v+d/2}},
$$
where we suppress the dependence on $v$ and $\sigma^2$ on the left hand sides.
\end{enumerate}
\noindent
\\
\\
Posterior contraction rate of stationary Gaussian processes is partly determined by the tail behavior of its spectral density. In the rest of this appendix, we establish some ancillary results and some useful properties of  the spectral density of the CH covariance function, needed in the proofs of the main theorems.

Let $\Gamma(x),\; x\in \mathbb{R}^{+}$ denote the gamma function for a positive real-valued argument. The lower and upper incomplete gamma functions are defined respectively as:
$$
\gamma(a, x)=\int_0^x e^{-t} t^{a-1} dt; \qquad  \Gamma(a, x)=\int_x^{\infty} e^{-t} t^{a-1} dt, \qquad a>0.
$$
A useful inequality \citep{alzer1997some,gautschi1998incomplete} for the incomplete gamma function is:
\begin{equation}  \label{incomplete gamma inequality}
\quad\left(1-e^{-s_a x}\right)^a<\frac{\gamma(a, x)}{\Gamma(a)}<\left(1-e^{-r_a x}\right)^a, \quad 0 \leq x<\infty, \quad a>0, \quad a \neq 1,
\end{equation}
where,
$$
r_a=\left\{\begin{array}{ll}
{[\Gamma(1+a)]^{-1 / a}} & \text { if } 0<a<1, \\
1 & \text { if } a>1,
\end{array} \quad s_a= \begin{cases}1 & \text { if } 0<a<1, \\
{[\Gamma(1+a)]^{-1 / a}} & \text { if } a>1.\end{cases}\right.
$$
\begin{lemma}  \label{incomplete asymp}
We have,
\begin{equation}
\lim_{x\to\infty}\frac{\gamma(x+1, x)}{\Gamma(x+1)} = 1/2. \label{incomplete gamma go to 1}
\end{equation}
\end{lemma}
\begin{proof}[Proof of Lemma \ref{incomplete asymp}]
(At the time of writing, a sketch of the proof is available at: \href{https://math.stackexchange.com/questions/3751528/limits-of-the-incomplete-gamma-function}{math.stackexchange.com}, which we reproduce below, unable to locate a persistent citable academic item.)  Let $t=x+u \sqrt{x}$. Then,
\begin{align}
\Gamma(x+1, x) &=\int_x^{\infty} t^x e^{-t} \mathrm{~d} t 
=x^{x+\frac{1}{2}} e^{-x} \int_0^{\infty}\left(1+\frac{u}{\sqrt{x}}\right)^x e^{-\sqrt{x} u} \mathrm{~d} u.\label{eq:gam1}
\end{align}
Next, note that:
$$
\lim _{x \rightarrow \infty}\left(1+\frac{u}{\sqrt{x}}\right)^x e^{-\sqrt{x} u}=e^{-\frac{u^2}{2}}.
$$
Applying  the  inequality $\log (1+x) \leq x-\frac{x^2}{2(x+1)}$ for $x \geq 0$ shows that, 
$$
\left(1+\frac{u}{\sqrt{x}}\right)^x e^{-\sqrt{x} u} \leq e^{-\frac{u^2}{2(u+1)}},
$$
for all $x \geq 1$ and $u \geq 0$. Since this bound is integrable on $[0, \infty)$,  by the dominated convergence theorem, 
\begin{align}
\lim _{x \rightarrow \infty} \int_0^{\infty}\left(1+\frac{u}{\sqrt{x}}\right)^x e^{-\sqrt{x} u} \mathrm{~d} u &=\int_0^{\infty} e^{-\frac{u^2}{2}} \mathrm{~d} u=\sqrt{\frac{\pi}{2}}.\label{eq:gam2}
\end{align}
An application of Stirling's formula yields:
\begin{align}
\Gamma(x+1) &\sim \sqrt{2 \pi} x^{x+\frac{1}{2}} e^{-x}, \quad \text { as } \quad x \rightarrow \infty.\label{eq:gam3}
\end{align}
Combining \eqref{eq:gam1}, \eqref{eq:gam2} and \eqref{eq:gam3}, we obtain,
$$
\lim _{x \rightarrow \infty} \frac{\Gamma(x+1, x)}{\Gamma(x+1)}=\frac{1}{2}.
$$
Noting that $\gamma(x+1,x) + \Gamma(x+1,x)=\Gamma(x+1)$ completes the proof. 
\end{proof}

\begin{lemma}
Let $\{a_n\}, \{b_n\}>0$ be sequences such that  $a_n=O(1)$. Then, we have, as $n\to\infty$,
\begin{displaymath}
    \int_{0}^{a_n} x^{b_n-1} \exp(-x) dx \asymp a_n^{b_n}/b_n.
\end{displaymath}
\label{incomplete integrate}
\end{lemma}
\begin{proof}[Proof of Lemma \ref{incomplete integrate}]
For the upper bound, we have,
\begin{displaymath}
    \int_{0}^{a_n} x^{b_n-1} \exp(-x) dx \leq  \int_{0}^{a_n} x^{b_n-1} dx =a_n^{b_n}/b_n .
\end{displaymath}
For the lower bound, 
\begin{displaymath}
    \int_{0}^{a_n} x^{b_n-1} \exp(-x) dx  \ge \exp(-a_n)\int_{0}^{a_n} x^{b_n-1} dx = \exp(-a_n)\cdot{a_n}^{b_n} /{b_n} \gtrsim {a_n}^{b_n} /{b_n}. 
\end{displaymath}
\end{proof} 

The next lemma obtains the upper and lower bounds for the spectral density of the CH class.
\begin{lemma} \label{inequality about ch spectral density}
If $\alpha>d/2+1$ and $\beta^2=O(1)$  as $n \to \infty$, then,
$$
\frac{\Gamma(\alpha-d/2)}{\Gamma(\alpha)}\beta^{d} \lesssim (1+\lambda^2)^{(v+d/2)}m^{\alpha,\beta}_{CH}(\lambda) \lesssim \frac{\Gamma(\alpha +v)}{\Gamma(\alpha )\beta^{2v}}.
$$
\end{lemma}

\begin{proof}  [Proof of Lemma \ref{inequality about ch spectral density}] Let  
$h=\beta^2/(2\phi^2)$. Then,
\begin{displaymath}
\begin{split}
&(1+\lambda^2)^{(v+d/2)}m^{\alpha,\beta}_{CH}(\lambda) \\
=&  \frac{\sigma^2 2^{v-\alpha} v^v \beta^{2\alpha}}{\pi^{d/2}\Gamma(\alpha)}(1+\lambda^2)^{(v+d/2)}\int_0^{\infty} (2v\phi^{-2}+\lambda^2)^{-v-d/2}  \phi^{-2(v+\alpha+1)} \exp{(-\beta^2/(2\phi^2))d\phi^2}  \\
           =& \frac{\sigma^2 2^{v-\alpha} v^v \beta^{2\alpha}}{\pi^{d/2}\Gamma(\alpha)} 
           \left[\int_0^{1}  \left(\frac{1+\lambda^2}{2v+\phi^2\lambda^2}\right)^{(v+d/2)}\phi^{-2(-d/2+\alpha+1)} \exp{(-\beta^2/(2\phi^2))}d\phi^2\right.\\
           &+\left.\int_1^{\infty}  \left(\frac{1+\lambda^2}{2v\phi^{-2}+\lambda^2}\right)^{(v+d/2)}\phi^{-2(v+\alpha+1)} \exp{(-\beta^2/(2\phi^2))}d\phi^2\right] \\  
           \gtrsim& \frac{\sigma^2 2^{v-\alpha} v^v \beta^{2\alpha}}{\pi^{d/2}\Gamma(\alpha)} \left[\left(\frac{\beta^2}{2}\right)^{d/2-\alpha}\int_{\frac{\beta^2}{2}}^{\infty}h^{\alpha-d/2-1}\exp(-h) dh \right. \\
           &+\left. \left(\frac{\beta^2}{2}\right)^{-v-\alpha}\int^{\frac{\beta^2}{2}}_{0}h^{\alpha+v-1}\exp(-h) dh \right]\\
           \gtrsim& \frac{1}{\Gamma(\alpha)} \left[\left(\frac{\beta^2}{2}\right)^{d/2}\int_{\frac{\beta^2}{2}}^{\infty}h^{\alpha-d/2-1}\exp(-h) dh+
           \left(\frac{\beta^2}{2}\right)^{-v}\int^{\frac{\beta^2}{2}}_{0}h^{\alpha+v-1}\exp(-h) dh\right].
\end{split}
\end{displaymath}
By Lemma \ref{incomplete integrate} and Stirling’s approximation, we have 
$
    \int^{{\beta^2}/{2}}_{0}h^{\alpha+v-1}\exp(-h) dh $ $\asymp \frac{({\beta^2}/{2})^{\alpha+v}}{\alpha+v},
$
and,
$
    \int_{{\beta^2}/{2}}^{\infty}h^{\alpha-d/2-1}\exp(-h) dh \asymp  \Gamma(\alpha-d/2),
$
yielding the lower bound. The upper bound follows similarly. 
\end{proof}
The next lemma gives an alternative lower bound for  the spectral density of CH class, depending on the relationship between $\alpha$ and $\beta$.
\begin{lemma} \label{CH spectral density tail behavior}
Suppose $\beta^2=O(1)$ and $\alpha>d/2+1$, for $\alpha$ fixed or tending to infinity as $n \to \infty$. Then,
$$
(1+\lambda^2)^{v+d/2} m^{\alpha,\beta}_{CH}(\lambda) \gtrsim \left\{\begin{array}{lll}
\frac{\Gamma(\alpha-d/2)\beta^{d}}{\Gamma(\alpha)}(1+\lambda^{2})^{v+d/2},  & \text { if } \lambda^2 \beta^{2} \leq 1, \\
\frac{\beta^{-2v}\Gamma(\alpha-d/2)}{\Gamma(\alpha)}\frac{\alpha}{e^\alpha}, & \text { if } 1 < \lambda^2 \beta^{2} < \alpha+v-1,\\
\frac{\beta^{-2v}\Gamma(\alpha+v)}{\Gamma(\alpha)},& \text { if } \lambda^2 \beta^{2} \geq  \alpha+v-1.
\end{array}\right.
$$
\end{lemma}
\begin{proof}[Proof of Lemma \ref{CH spectral density tail behavior}] 
We only prove the $\alpha \to \infty$ case. The fixed $\alpha$ case can be proved by the same method. Let $ h=\beta^2/(2\phi^2)$. Then,
\begin{equation}\label{spectral density expansion}
\begin{split}
&(1+\lambda^2)^{(v+d/2)}m^{\alpha,\beta}_{CH}(\lambda) \\
=&  \frac{\sigma^2 2^{v-\alpha} v^v \beta^{2\alpha}}{\pi^{d/2}\Gamma(\alpha)}\int_0^{\infty} \left(\frac{1+\lambda^2}{2v\phi^{-2}+\lambda^2}\right)^{v+d/2}  \phi^{-2(v+\alpha+1)} \exp{(-\beta^2/(2\phi^2))d\phi^2}  \\
=&  \frac{\sigma^2  {(4v)}^v \beta^{-2v}}{\pi^{d/2}\Gamma(\alpha)}\int_0^{\infty} \left(\frac{1+\lambda^{2}}{4vh \beta^{-2}+\lambda^{2}}\right)^{v+d/2}  h^{v+\alpha-1} \exp(-h) dh\\
=&  \frac{\sigma^2  {(4v)}^v \beta^{-2v}}{\pi^{d/2}\Gamma(\alpha)} \left[\int_0^{\lambda^2\beta^2} \left(\frac{1+\lambda^{2}}{4vh \beta^{-2}+\lambda^{2}}\right)^{v+d/2}  h^{v+\alpha-1} \exp(-h) dh\right.
\\
&+\left.\int_{\lambda^2\beta^2}^{\infty} \left(\frac{1+\lambda^{2}}{4vh \beta^{-2}
+\lambda^{2}}\right)^{v+d/2}  h^{v+\alpha-1} \exp(-h) dh \right]\\
\asymp& \frac{\beta^{-2v}}{\Gamma(\alpha)} \left[\int_0^{\lambda^2\beta^2} \left(\frac{1+\lambda^{2}}{4vh \beta^{-2}+\lambda^{2}}\right)^{v+d/2}  h^{v+\alpha-1} \exp(-h) dh\right.
\\
&+\left.[(1+\lambda^{2})\beta^{2}]^{v+d/2}\int_{\lambda^2\beta^2}^{\infty} h^{\alpha-d/2-1} \exp(-h) dh \right].
\end{split}
\end{equation}
When $\lambda^2 \beta^{2}\leq 1$, by Lemma \ref{incomplete integrate}:
\begin{displaymath}
\begin{split}
&(1+\lambda^2)^{(v+d/2)}m^{\alpha,\beta}_{CH}(\lambda) \\
&\gtrsim \frac{\beta^{-2v}}{\Gamma(\alpha)} [(1+\lambda^{2})\beta^{2}]^{v+d/2}\int_{\lambda^2\beta^2}^{\infty} h^{\alpha-d/2-1} \exp(-h) dh \\
&\asymp \frac{\beta^{-2v}}{\Gamma(\alpha)}[(1+\lambda^{2})\beta^{2}]^{v+d/2} \Gamma(\alpha-d/2)\\
&\asymp \frac{\Gamma(\alpha-d/2)\beta^{d}}{\Gamma(\alpha)}(1+\lambda^{2})^{v+d/2}. 
\end{split}
\end{displaymath}
When $ \lambda^2 \beta^{2} \geq \alpha+v-1 $, we have, by (\ref{incomplete gamma go to 1}) and (\ref{spectral density expansion}):
\begin{displaymath}
\begin{split}
&(1+\lambda^2)^{(v+d/2)}m^{\alpha,\beta}_{CH}(\lambda) \\
&\asymp \frac{\beta^{-2v}}{\Gamma(\alpha)} \left[\int_0^{\lambda^2\beta^2}  h^{v+\alpha-1} \exp(-h) dh
+[(1+\lambda^{2})\beta^{2}]^{v+d/2}\int_{\lambda^2\beta^2}^{\infty} h^{\alpha-d/2-1} \exp(-h) dh \right]\\
&\gtrsim \frac{\beta^{-2v}\Gamma(\alpha+v)}{\Gamma(\alpha)}.
\end{split}
\end{displaymath}
When $ 1\leq \lambda^2 \beta^{2} \leq \alpha+v-1 $, we have, by (\ref{incomplete gamma inequality}), (\ref{spectral density expansion})  and $e^x \ge 1+x$:
\begin{displaymath}
\begin{split}
&(1+\lambda^2)^{(v+d/2)}m^{\alpha,\beta}_{CH}(\lambda) \\
\asymp& \frac{\beta^{-2v}}{\Gamma(\alpha)} \left[\int_0^{\lambda^2\beta^2}  h^{v+\alpha-1} \exp(-h) dh
+[(1+\lambda^{2})\beta^{2}]^{v+d/2}\int_{\lambda^2\beta^2}^{\infty} h^{\alpha-d/2-1} \exp(-h) dh \right]\\
\gtrsim& \frac{\beta^{-2v}}{\Gamma(\alpha)}[(1+\lambda^{2})\beta^{2}]^{v+d/2}\int_{\lambda^2\beta^2}^{\infty} h^{\alpha-d/2-1} \exp(-h) dh\\
\gtrsim& \frac{\beta^{-2v}}{\Gamma(\alpha)}\int_{ \alpha+v-1}^{\infty} h^{\alpha-d/2-1} \exp(-h) dh\\
\geq& \frac{\beta^{-2v}\Gamma(\alpha-d/2)}{\Gamma(\alpha)}[1-(1-e^{- \alpha-v+1})^{\alpha-d/2}]\\
\gtrsim&\frac{\beta^{-2v}\Gamma(\alpha-d/2)}{\Gamma(\alpha)}\left[1-\exp\left(-{ \frac{\alpha-d/2}{e^{ \alpha+v-1}}}\right)\right]\\
\gtrsim& \frac{\beta^{-2v}\Gamma(\alpha-d/2)}{\Gamma(\alpha)}{ \frac{\alpha-d/2}{e^{ \alpha+v-1}}}\\
\gtrsim&\frac{\beta^{-2v}\Gamma(\alpha-d/2)}{\Gamma(\alpha)}\frac{\alpha}{e^\alpha}. 
\end{split}
\end{displaymath}
\end{proof}

\noindent
We denote the unit ball of RKHS $\mathbb{H}$ by $\mathbb{H}_1$.
\begin{lemma} \label{rescaled hilbert inclusion}
    Assume the spectral density $m(\lambda)$ satisfies  that $a \to m(a \lambda)$ is decreasing on $(0, \infty)$ for every $\lambda \in \mathbb{R}^d$. If $a \leq b$, then for Mat\'ern process we have $\frac{1}{b^{d/2}}\mathbb{H}_1^b \subset \frac{1}{a^{d/2}}\mathbb{H}_1^a$ and for CH process we have $\frac{1}{b^{d/2}}\mathbb{H}_1^{\alpha,b} \subset \frac{1}{a^{d/2}}\mathbb{H}_1^{\alpha,a}$, where $\alpha>0$ is any fixed number.
\end{lemma}
\begin{proof}
Here we only prove the Mat\'ern process case, and the proof of the CH case is the same.

We have $m_M^b/m_M^a(\lambda)=(b/a)^d [m_M^1(b \lambda) /m_M^1(a \lambda) ]\leq (b/a)^d$. Then by Lemma \ref{rescaled measure}, an arbitrary element of $\mathbb{H}_1^b$ has the form:
$$
\mathcal{F}_b h= \int e^{i <\lambda,t>} h(\lambda) d \mu_b(\lambda)=\int e^{i <\lambda,t>} \left(h\frac{m_M^b}{m_M^a}\right)d \mu_a(\lambda),
$$
where $h \in L_2(\mu_b)$. Let $g$ be the  smallest choice of minimum norm  in (\ref{RKHS of rescaled matern}) for $\|\mathcal{F}_b h\|_{\mathbb{H}^b}$,
and let $\tilde{h}$ be the  smallest choice of minimum norm  in (\ref{RKHS of rescaled matern}) for $\|\mathcal{F}_b g\|_{\mathbb{H}^a}$. Then,
$$
\int |h\frac{m_M^b}{m_M^a}|^2d \mu_a(\lambda) \leq \|\frac{m_M^b}{m_M^a}\|_{\infty}\int |h|^2 d \mu_b(\lambda) \leq (b/a)^{d}\int |h|^2d \mu_b(\lambda).
$$
Then we have, 
\begin{displaymath} 
\begin{split}
&\|\mathcal{F}_b h\|^2_{\mathbb{H}^a}=\|\mathcal{F}_b g\|^2_{\mathbb{H}^a}=
\|\mathcal{F}_a \Tilde{h}\|^2_{\mathbb{H}^a}=
\int |g\frac{m_M^b}{m_M^a}|^2d \mu_a(\lambda)
\\
\leq &\|\frac{m_M^b}{m_M^a}\|_{\infty}\int |g|^2 d \mu_b(\lambda) \leq (\frac{b}{a})^d \|\mathcal{F}_b h\|^2_{\mathbb{H}^b}.
\end{split}
\end{displaymath}
This finishes the proof.


\end{proof}

\begin{lemma} \label{uniform distance within hilbert}
 For any $h \in \mathbb{H}_1^a$ for Mat\'ern process(or $\mathbb{H}_1^{\alpha,a}$ for CH process) and $t \in \mathbb{R}^d$, we have $|h(0)|^2 \leq \|\mu\|=\int d\mu$ and $|h(t)-h(0)| \leq a^{-1} \|t\| (\int \|\lambda\|^2 d \mu)^{1/2}$, where $\mu$ is the spectral measure with rescaling parameter equal to $1$.
\end{lemma}
\begin{proof}
Here also we only prove the Mat\'ern process case, and the proof of the CH case is the same.

By Lemma \ref{rescaled measure},  if $h \in \mathbb{H}_1^a$, then there exists a function $\psi$ such that $\mathcal{F}_a \psi=h$ and $\int |\psi|^2 d\mu_a \leq 1$. Then by the same method of the proof of Lemma 4.8 in \cite{van2009adaptive}, the proof follows.
\end{proof}

\end{appendix}

\begin{acks}[Acknowledgments]
The authors would like to thank the anonymous referees, an Associate
Editor and the Editor for their constructive comments that improved the presentation of the main results.
\end{acks}

\begin{funding}
The research of Bhadra was partially supported by U.S. National Science Foundation Grants DMS-2014371 and SES-2448704.
\end{funding}

\bibliographystyle{apalike}
\bibliography{samplebib}    
\end{document}